\documentclass[hidelinks,onefignum,onetabnum]{siamart251216}

\usepackage{wasysym} 
\usepackage{enumerate}
\usepackage{amsfonts}
\usepackage{graphicx}
\usepackage{subcaption} 
\usepackage{epstopdf}
\usepackage[dvipsnames]{xcolor}
\usepackage{algorithm}
\usepackage{algorithmic}

\usepackage[normalem]{ulem}

\usepackage{multirow}

\ifpdf
  \DeclareGraphicsExtensions{.eps,.pdf,.png,.jpg}
\else
  \DeclareGraphicsExtensions{.eps}
\fi

\usepackage{booktabs}

\newcommand{\R}{\mathbb R}

\newcommand{\mcM}{\mathcal M}

\newcommand{\St}{\operatorname{St}}

\newcommand{\argmin}{\operatorname*{arg\,min}}

\newcommand{\sym}{\operatorname{sym}}
\newcommand{\Skew}{\operatorname{skew}} 
\newcommand{\SPD}{\operatorname{SPD}}
\newcommand{\dist}{\operatorname{dist}}
\newcommand{\Exp}{\operatorname{Exp}}
\newcommand{\Log}{\operatorname{Log}}
\newcommand{\Cay}{\operatorname{Cay}}
\newcommand{\id}{\operatorname{id}}

\DeclareMathOperator{\diag}{diag}
\makeatletter
\renewcommand*\env@matrix[1][*\c@MaxMatrixCols c]{%
  \hskip -\arraycolsep
  \let\@ifnextchar\new@ifnextchar
  \array{#1}}
\makeatother

\newcount\Comments  
\newcommand{\kibitz}[2]{\ifnum\Comments=1\textcolor{#1}{#2}\fi}

\newcommand{\rev} {\textcolor{black}}

\newsiamremark{remark}{Remark}
\newsiamremark{hypothesis}{Hypothesis}
\crefname{hypothesis}{Hypothesis}{Hypotheses}
\newsiamthm{claim}{Claim}
\newsiamremark{fact}{Fact}
\crefname{fact}{Fact}{Facts}

\headers{A new polar factor retraction}{R. Jensen, and R. Zimmermann}

\title{A new polar factor retraction on the Stiefel manifold with closed-form inverse\thanks{Submitted to the editors DATE.
\funding{This work was supported by the Independent Research Foundation Denmark, DFF, grant nr. 3103-00094B}}}

\author{Rasmus Jensen\thanks{Department of  Mathematics and Computer Science, SDU Odense 
  (\email{rasmusj@imada.sdu.dk}, \email{zimmermann@imada.sdu.dk}).}
\and Ralf Zimmermann\footnotemark[2]}

\usepackage{amsopn}

\ifpdf
\hypersetup{
  pdftitle={A new Stiefel retraction with closed-form inverse},
  pdfauthor={R. Jensen, and Ralf Zimmermann}
}
\fi

\begin{document}

\maketitle

\begin{abstract}
Retractions are the workhorses in Riemannian computing applications, where computational efficiency is of the essence. This work introduces a new retraction on the compact Stiefel manifold of orthogonal frames. The retraction is second-order accurate under the Euclidean metric and features a closed-form inverse that can be efficiently computed. 

A variety of retractions is known on the Stiefel manifold, including the Riemannian exponential map, the polar factor retraction, the QR-retraction, quasi--geodesics and the Cayley retraction. The Cayley retraction is second--order accurate under the canonical metric and features a closed-form inverse.
The new retraction is the first one with the corresponding features under the Euclidean metric.

We present numerical experiments which illustrates the properties of the new retraction, as well as compare it to numerous of the currently available alternatives. In addition, we examine the performance of the retraction when used for interpolation and for computing a Riemannian barycenter. 
\end{abstract}



\section{Introduction}
Practical data processing on manifolds requires local coordinates, which make it possible to map data `there and back':
‘There’ means mapping data from a Euclidean coordinate domain, for example on the tangent space, to the manifold.
‘Back’ refers to the reverse action of mapping manifold data into a coordinate domain.
In certain applications, e.g., Riemannian optimization, only the `there'-direction is needed, but for manifold interpolation or for computing Riemannian averages, there must also be an efficient way to go `back'.

This work focuses on the Stiefel manifold of orthogonal frames. Popularized by \cite{EdelmanAriasSmith1999}, the Stiefel manifold features in a huge amount of applications in deep learning \cite{Massart:2023},
 computer vision \cite{Ma:1998,Lui:2012,Turaga_2011}, statistics, signal processing and clustering \cite{Cai2023,Chakraborty2018,Chen2021,Cichocki2002,Kaneko:2013,Pennec2006,Pealat2023,Tian2021}, stochastic differential equations \cite{Marjanovic2016,Marjanovic2017}, and general numerical linear algebra \cite{boumal2015rtrmcextended,sato2021riemannian}, and the amount of literature is growing. 

Structure--preserving interpolation of column-orthonormal matrices, hence interpolation on the Stiefel manifold, has mainly been applied in the context of parametric model-order reduction \cite{ ElOmari:2025,Friderikos2021,Zimmermann2018}. The task is, given measurements of the system at selected parameters, to produce a low--rank basis representing the dynamics of the system under varying parameter configurations. 

Classically, interpolation tasks on manifolds have been carried out using the Riemannian normal coordinates, i.e., the Riemannian exponential and logarithm maps, which enjoy desirable geometric features. The Riemannian logarithm maps the sampled data to the tangent space at a designated point (the `back'), interpolation takes place in this linear space, and the interpolant is mapped back to the manifold via the Riemannian exponential (the `there'). 

In applications where computational efficiency is essential, one often employs approximations of the Riemannian exponential map, called retractions. By definition, all retractions match the Riemannian exponential up to terms of (at least) first--order. Thus, they are local diffeomorphisms, and their inverses approximate the Riemannian logarithm. 

Closed--form formulas for inverse retractions are relevant for vector transport, i.e., mapping vectors from one tangent space to another. For example, this is required  in the Riemannian conjugate--gradient methods \cite{Zhu2020}. Inverse retractions are also needed for  Riemannian interpolation \cite{jensen2025maxvol,SeguinKressner:2024}. 

A {\em second--order retraction} is a retraction that matches the Riemannian exponential map up to terms of second order. Employing such retractions simplifies the convergence analysis of Riemannian trust--region methods \cite[Section 6.4]{boumal2023}, since they preserve, in a certain sense, the Riemannian Hessian \cite[Proposition 5.45]{boumal2023}. Using higher--order retractions also increases regularity of subdivision schemes \cite{Duchamp:2013,Yazdani:2011}, and often they allow for larger time steps when solving ODEs on manifolds \cite{Gawlik:2018}.

Various retractions are known on the Stiefel manifold; we are aware of:
\begin{enumerate}[(1)]
    \item the Riemannian exponential map 
    \cite{EdelmanAriasSmith1999},
    \item  the polar factor retraction \cite[Section 4]{AbsilMahonySepulchre2008},
    \item the QR retraction \cite[Section 4]{AbsilMahonySepulchre2008},
    \item the Cholesky QR-based retraction \cite{Sato_CholeskyRetraction2019},
    \item the Cayley retraction \cite{Wen:2013}.
    \item the quasi-geodesic retractions of \cite{Bendokat_quasiGeo2021}.
\end{enumerate}   

The Riemannian exponential map (1) is the reference when quantifying the order of a retraction. 
It depends on the chosen metric \cite{HueperMarkinaLeite2020}. Computing the inverse requires an iterative algorithm \cite{  MataigneStiefelLog:2025,Sutti_Stiefel:2024,StiefelLog_Zimmermann2017}.
The polar factor retraction (2) is of second order under the Euclidean metric. The QR-retraction (3) is of first order under any metric, and so is its Cholesky counterpart (4).
Computing the inverse maps for the retractions (2) and (3) is based on solving a Sylvester/Riccati/Lyapu\-nov-type matrix equation \cite{Kaneko:2013}.

For the Cholesky QR-based retraction, we are not aware of any published work on computing the inverse. However, as with (2) and (3), a matrix equation is expected to arise in this task. 

To the best of our knowledge, the Cayley retraction (5) and the quasi-geodesic retractions of (6) are the only ones on the list that feature a closed form inverse \cite{Zhu2020,Bendokat_quasiGeo2021}. These retractions are of order $1$ under the Euclidean metric, and it turns out that the Cayley retraction is second--order accurate under the canonical metric.

In this work, we introduce a new retraction on the Stiefel manifold that is based on an additional twist (quite literally) in the polar factor retraction. 
The main features of this retraction, which we call {\em polar-light retraction}, are: 
\begin{itemize}
\item It is second-order accurate under the Euclidean metric.
\item It has a closed-form inverse.
\item Evaluating its inverse incurs asymptotically the same computational cost as evaluating the retraction itself. Computing matrix functions and decompositions is necessary only for small $(p\times p)$ matrices; with large $(n\times p)$ matrices, only simple matrix–matrix multiplications arise.
\item The closed formulas are analytic and do not formally require matrix decompositions.\footnote{when stated as in the upcoming equations \eqref{eq:varphi_U_Onp2}, \eqref{eq:psi_U_Onp2}} This is an advantage when derivatives are needed.
\item Unlike with the Cayley retraction, with appropriate matrix decompositions as an upfront investment, the associated retraction curves can be parameterized for efficient multi‑query use. 
\end{itemize}
Our experiments show that the polar--light retraction is usually a better approximation to the Riemannian exponential map than the polar factor retraction. In special cases, it reproduces the Riemann exponential exactly. {\color{black} We would like to highlight that the polar-light retraction shouldn't be understood as a light version of the polar factor retraction in terms of computational costs, nor is it an approximation of the latter. The word `\textit{light}' comes from the fact that we only need to compute matrix decompositions of ($p\times p$) matrices when computing its inverse. This is in contrast to the inverse polar-light retraction, which requires solving an equation system.}


{\em Organization:} Stiefel manifold essentials and the classical polar factor retraction are recapped in \Cref{sec:background}. The  theoretical findings constitute \Cref{sec:main_theoretical_findings}. We will present several numerical experiments in \Cref{sec:experiments}, investigating the behavior of the polar--light retraction in comparison with the polar factor retraction, as well as with several of the previously mentioned alternatives. We conclude the paper in \Cref{sec:conclusions}.

%

\section{Background}\label{sec:background}
The orthogonal group and special orthogonal group are
\[
    O(n) = \{\mathbf{Q}\in{\mathbb R}^{n\times n}\mid \mathbf{Q}^T\mathbf{Q}=I_{n}\},
    \text{ and }
    SO(n) = \{\mathbf{Q}\in O(n) \mid \det(\mathbf{Q})=1\}.
\]
The Stiefel manifold is their rectangular relative
\[
	\St(n,p) = \{U\in{\mathbb R}^{n\times p}\mid U^TU=I_{p}\}.
\]
Here and throughout, we write $I_p$ for the $(p\times p)$-identity matrix. We single out the special Stiefel point
$$
 E:= \begin{bmatrix}
	I_p\\
	0
\end{bmatrix}\in \St(n,p).
$$
The dimension of the Stiefel manifold is $\frac{p}2(p-1)+(n-p)p$, which reflects the number of independent parameters in a skew-symmetric ($p\times p$)-matrix and a rectangular $((n-p)\times p)$-matrix.
Fix $\hat U\in \St(n,p)$ with orthogonal completion $\hat U_{\bot} \in \St(n,n-p)$ so that $\widehat{\mathbf{Q}}=\begin{bmatrix}
  \hat U & \hat U_{\bot} 
\end{bmatrix}\in O(n)$. The tangent space at $\hat U$ is 
$$
    T_{\hat U}\St(n,p)=\{\xi = \widehat{\mathbf{Q}}\left[\begin{smallmatrix}
    A\\
    B
\end{smallmatrix}\right]\mid A\in \mathrm{skew}(p), B\in{\mathbb R}^{(n-p)\times p)} \}
=\widehat{\mathbf{Q}}T_E\St(n,p),
$$
\rev{where $T_{E}\St(n,p)= \Skew(p)\times {\mathbb R}^{(n-p)\times p)}\cong{\mathbb R}^{\frac{p}2(p-1)+(n-p)p}$.}
For details, see \cite{AbsilMahonySepulchre2008,EdelmanAriasSmith1999}. {\color{black}The Stiefel manifold admits a one-parameter family of Riemannian metrics \cite{HueperMarkinaLeite2020}
\begin{equation*}
\langle\xi_1, \xi_2\rangle_\beta = \beta \text{Tr}(A_1^T A_2) + \text{Tr}(B_1^T B_2), \quad \forall \xi_1, \xi_2 \in T_{\hat{U}}\text{St}(n,p).
\end{equation*}
The considerations in this paper focus primarily on the special cases $\beta=1$ or $\beta  = \frac{1}{2}$, which corresponds to the Euclidean and canonical metric, respectively.}

We denote the set of symmetric positive definite ($p\times p$)-matrices by 
$\SPD(p)$.
\subsection{The classical Stiefel polar factor retraction}
With the help of the polar decomposition, any rectangular real matrix $V\in\R^{n\times p}$ with full rank can be expressed as the unique product of the symmetric positive definite matrix square root $S=\sqrt{V^TV}$
and the column-orthogonal --hence Stiefel-- matrix $U=VS^{-1}$,
\[
    V = US, \quad U\in \St(n,p), \quad S\in \SPD(p),
\]
\cite[Theorem 8.1]{Higham:2008:FM}.
This gives rise to the {\em polar factor retraction}:
\begin{equation}
\label{eq:classic_PFR}
	R_{\hat U}^{\textsf{PF}}: T_{\hat U}\St(n,p)\to \St(n,p), \xi= \widehat{\mathbf{Q}}\left[\begin{smallmatrix}
    A\\
    B
\end{smallmatrix}\right] \mapsto
	(\hat U+\xi)\left(I_p + A^TA+B^TB\right)^{-\frac12},
\end{equation}
see \cite[eq. (4.7)]{AbsilMahonySepulchre2008}.
Note that 
$
    \hat U + \xi = \widehat{\mathbf{Q}}(E+ \left[\begin{smallmatrix}
    A\\
    B
\end{smallmatrix}\right])
$
so that 
\[
    (\hat U + \xi)^T(\hat U + \xi) = I_p + A^TA+B^TB = I_p+\xi^T\xi.
\]
Hence, $R_{\hat U}^{\textsf{PF}}$ sends the (non-orthogonal) $(n\times p)$-matrix $\hat U + \xi$ to the Stiefel manifold by mapping it to the orthogonal factor of its full polar decomposition. For the sake of argument, we call it the {\em full polar factor retraction}.
Inversion requires solving a Lyapunov equation \cite{Kaneko:2013}.
Higher-order extensions are discussed in \cite{Gawlik:2018}. 



\section{The polar-light retraction}\label{sec:main_theoretical_findings}
Any Stiefel matrix can be split into sub-blocks
$U= \begin{bmatrix}
	U_1\\
	U_2
\end{bmatrix} 
$
with $U_1\in {\mathbb R}^{p\times p}$ and $U_2\in{\mathbb R}^{(n-p)\times p}$.
If invertible, a local coordinate representation can be obtained from a polar decomposition of the small block $U_1\rev{\in\R^{p\times p}}$ alone \rev{instead of the full $U\in\R^{n\times p}$}; hence the term `{\em light}'.
\rev{Recall that the spectrum $\sigma(A)$ of a skew-symmetric matrix $A$ is purely imaginary.
\begin{lemma}[Polar-light coordinates around canonical center]
\label{lem:polar-light-chart}
Consider the special point  $ E= \begin{bmatrix}
	I_p\\
	0
\end{bmatrix} \in \St(n,p)$.
Let $\Skew_{\pi}(p):= \{A\in\Skew(p) \mid \sigma(A) \subset \{iy\mid y\in (-\pi,\pi)\}$.
The map
\begin{equation}
    \label{eq:novel_PFR}
	\varphi_E: \Skew_{\pi}(p)\times {\mathbb R}^{(n-p)\times p}
	\to \St(n,p),
	\begin{bmatrix}
		A\\
		B
	\end{bmatrix}\mapsto
	\begin{bmatrix}
		\exp_m(A)\\
		B
	\end{bmatrix}(I_{p} + B^TB)^{-\frac12}
\end{equation}
is a local paramaterization of the Stiefel manifold with coordinate center $\begin{bmatrix}
		0\\
		0
	\end{bmatrix}$ correponding to $\varphi_E(\begin{bmatrix}
		0\\
		0
	\end{bmatrix}) = E\in \St(n,p)$.\\
The image $\mathcal{B}:= \varphi_E\left(\Skew_{\pi}(p)\times {\mathbb R}^{(n-p)\times p}\right)\subset \St(n,p)$ is a
 (relative) open, path-connected neighborhood around $E$.
The map $\varphi_E$ is bijective onto its image;
the inverse map
\begin{equation}
\psi_E:\mathcal{B} \to \Skew_{\pi}(p)\times {\mathbb R}^{(n-p)\times p}
\\
U=
\begin{bmatrix}
	U_1\\
	U_2
\end{bmatrix}\mapsto
\begin{bmatrix}
	\log_m\left(U_1(U_1^TU_1)^{-\frac12}\right)\\
	U_2(U_1^TU_1)^{-\frac12}
\end{bmatrix}
=:
\begin{bmatrix}
	A\\
	B
\end{bmatrix}
\label{eq:novel_PF_chart}
\end{equation}
is a coordinate chart.
\end{lemma}
}
\rev{
\begin{proof}
Both maps are differentiable.
Because $\Skew_{\pi}(p)$ is the open ball $B(0,\pi)=\{A\in\Skew(p)\mid \|A\|_2<\pi\}$ inside the vector space of skew-symmetric matrices, $\Skew_{\pi}(p)\times \R^{(n-p)\times p}$ is path-connected and so is the image $\mathcal{B}$ of this set under $\varphi_E$. 
We directly check the identities
\[
    \psi_E\circ \varphi_E = \mathrm{id}_{\text{skew}_{\pi}(p)\times {\mathbb R}^{(n-p)\times p}},\quad 
    \varphi_E\circ \psi_E = \mathrm{id}_{\mathcal{B}}.
\]
Let $\begin{bmatrix}
	A\\
	B
\end{bmatrix}\in\Skew_{\pi}(p)\times \mathbb{R}^{(n-p)\times p}$. The restriction of $A$ to $\Skew_{\pi}(p)$ ensures that $\exp_m(A)$ has no eigenvalue of $-1$, which in turn ensures that the principal matrix logarithm of $\exp_m(A)$ is well-defined and
$\log_m(\exp_m(A))=A$, see \cite[Thm 1.31]{Higham:2008:FM}.
The matrix $I_p+B^TB$ is symmetric positive definite so that $(I_{p} + B^TB)^{-\frac12}$ is well-defined.
It holds
\begin{align*}
    (\psi_E\circ \varphi_E)\begin{bmatrix}
	A\\
	B
\end{bmatrix}
&= \Psi_E\left(
\begin{bmatrix}
		\exp_m(A)(I_{p} + B^TB)^{-\frac12}\\
		B(I_{p} + B^TB)^{-\frac12}
	\end{bmatrix}\right)\\
    &= \begin{bmatrix}
		\log_m\left(\exp_m(A)(I_{p} + B^TB)^{-\frac12}
        \left(
        ((I_{p} + B^TB)^{-\frac12})^2
        \right)^{-\frac12}
        \right)\\
		B(I_{p} + B^TB)^{-\frac12}
        \left(
        ((I_{p} + B^TB)^{-\frac12})^2
        \right)^{-\frac12}
	\end{bmatrix}\\
    &= \begin{bmatrix}
		\log_m\left(\exp_m(A)
        \right)\\
		B
	\end{bmatrix}= \begin{bmatrix}
		A\\
		B
	\end{bmatrix}.
\end{align*}
For the other direction, let $U= \begin{bmatrix}
		U_1\\
		U_2
	\end{bmatrix}\in \mathcal{B}$. This ensures that $U_1$ is invertible and that there is $A\in\Skew_{\pi}(p)$, $B\in\mathbb{R}^{(n-p)\times p}$ such that $U_1 = \exp_m(A)(I_{p} + B^TB)^{-\frac12}$. Observe that this entails 
    $
    U_1^TU_1
    =  (I_{p} + B^TB)^{-1}
    $.
    Therefore
\begin{align*}
    (\varphi_E\circ \psi_E)\begin{bmatrix}
		U_1\\
		U_2
	\end{bmatrix}
    &= \varphi_E\left(
    \begin{bmatrix}
	\log_m\left(U_1(U_1^TU_1)^{-\frac12}\right)\\
	U_2(U_1^TU_1)^{-\frac12}
\end{bmatrix}
    \right)
    = \varphi_E\left(
    \begin{bmatrix}
	\log_m\left(\exp_m(A)\right)\\
	U_2(I_{p} + B^TB)^{\frac12}
\end{bmatrix}
    \right)\\
    &= \begin{bmatrix}
	\exp_m(A)\\
	U_2(I_{p} + B^TB)^{\frac12}
\end{bmatrix}
\left(I_p + (I_{p} + B^TB)^{+\frac12}U_2^T U_2(I_{p} + B^TB)^{+\frac12}\right)^{-\frac12}.
\end{align*}
Recall that the subblocks of a Stiefel matrix $U$ are related by $I_p= U_1^TU_1+ U_2^T U_2$.
As a consequence, the rightmost matrix factor reduces to
\begin{align*}
\left(I_p + (I_{p} + B^TB)^{+\frac12}
U_2^TU_2
(I_{p} + B^TB)^{+\frac12}\right)^{-\frac12}
=\\
\left(I_p + (I_{p} + B^TB)^{+\frac12}
(I_p- U_1^TU_1)
(I_{p} + B^TB)^{+\frac12}\right)^{-\frac12}\\
= \left(I_p + (I_{p} + B^TB)^{+\frac12}
(I_p- (I_{p} + B^TB)^{-1})
(I_{p} + B^TB)^{+\frac12}\right)^{-\frac12}\\
= \left(I_p + (I_{p} + B^TB)
-I_p\right)^{-\frac12}
= (I_{p} + B^TB)^{-\frac12}.
\end{align*}
This shows that eventually
\[
(\varphi_E\circ \psi_E)\begin{bmatrix}
		U_1\\
		U_2
	\end{bmatrix}
    = \begin{bmatrix}
	\exp_m(A)\\
	U_2(I_{p} + B^TB)^{+\frac12}
\end{bmatrix} (I_{p} + B^TB)^{-\frac12}
= \begin{bmatrix}
	U_1\\
	U_2
\end{bmatrix}.
\]
\end{proof}
}

The chart $\psi_E$ can be interpreted as follows:
A polar decomposition of the upper $(p\times p)$-block of $U$ yields a splitting 
\[
    U_1 = \underbrace{U_1(U_1^TU_1)^{-\frac12}}_{=:R\in O(p)}\underbrace{(U_1^TU_1)^{\frac12}}_{=:S\in\SPD(p)}.
\]
Because there is a continuous path from $E$ to $U$, and thus from the upper block $I_p$ to the upper block $U_1$, it holds
$\det(R)=\det(U_1(U_1^TU_1)^{-\frac12})=+1$, so that $R\in SO(p)$.
The following restrictions of the matrix exponential
\[
 \exp_m|_{\sym(p)}: \sym(p)\rightarrow  \SPD(p), \quad \exp_m|_{\Skew(p)}: \Skew(p)\rightarrow  SO(p)
\]
constitute a global diffeomorphism \cite[Thm. 2.8]{Pennec2006} and a globally surjective local diffeomorphism 
\cite[\S. 3.11, Thm. 9]{godement2017introduction}, respectively.
Hence, there is $A\in\Skew(p)$ and $X\in \sym(p)$ such that $R=\exp_m(A)$ and $S = \exp_m(X)$ and
\[
        U_1 =RS= \exp_m(A)\exp_m(X).\footnote{This splitting represents the $p^2$ degrees of freedom of a full-rank $(p\times p)$-matrix divided into $\frac12p(p-1)$ parameters for the skew-symmetric part and $\frac12p(p+1)$ parameters for the symmetric part in multiplicative form.}
\] 
Therefore, $\psi_E$ maps a Stiefel point
$\begin{bmatrix}
	U_1\\
	U_2
\end{bmatrix}$
from a suitable neighborhood of $E$ to the coordinate matrices
$$ A = \log_m(R), \quad B = U_2\exp_m(X)^{-1}.$$

The pattern underlying $\varphi_E$ {\color{black} that generalizes to coordinate centers other than $E$} is revealed by writing
\[
\varphi_E(\xi) =  
	\left(E \exp_m(E^T\xi) + (I_p-EE^T)\xi\right)
    (I_{p} + \xi^T (I-EE^T)\xi)^{-\frac12},
    \quad \xi = \begin{bmatrix}
		A\\
		B
	\end{bmatrix}.
\]
%
%
%
%
%
%
\subsection{Changing the coordinate center}
Suppose that $\hat U\in \St(n,p)$ is a point that we designate as the center of the coordinate chart.
Let $\widehat{\mathbf{Q}}= [\hat U \ \hat U_\perp]\in\mathrm{O}(n)$
be an arbitrary but fixed orthogonal completion.
Define
\begin{align}
\label{eq:varphi_U}
    \varphi_{\hat U}:~& T_{\hat U}\St(n,p) \to \St(n,p),
    & \xi\mapsto& \varphi_{\hat U} (\xi) :=\widehat{\mathbf{Q}}\varphi_{E}\left(\widehat{\mathbf{Q}}^T\xi\right),\\
\label{eq:psi_U}
    \psi_{\hat U}:~& \mathcal{B}_{\hat U}\to T_{\hat U}\St(n,p),
    & U\mapsto& \psi_{\hat U} (U) :=\widehat{\mathbf{Q}}\psi_{E}\left(\widehat{\mathbf{Q}}^TU\right).
\end{align}
Formally, these maps constitute a parameterization and a corresponding chart  with coordinate center $\hat U=\varphi_{\hat U}(0)$, $0=\psi_{\hat U}(\hat U)$, 

In practical computations, where $n\gg p$, it is infeasible to actually form the completion $\hat U_\bot$.
Fortunately, {\color{black} we can avoid this by the following standard technique}.
Consider a tangent vector 
$$
\xi = \widehat{\mathbf{Q}}\left[\begin{smallmatrix}
    A\\
    B
\end{smallmatrix}\right] = \hat U A + \hat U_\bot B = 
\hat U(\hat U^T\xi) + (I-\hat U\hat U^T)\xi.
$$
{\color{black}Using this expression in \eqref{eq:varphi_U}-\eqref{eq:psi_U} yields the following formulae} 
\begin{align}
    \varphi_{\hat U}(\xi) &= \left(\hat U\exp_m(\hat U^T\xi) + (I_p-\hat U\hat U^T)\xi\right)\left(I_p+\xi^T(I-\hat U\hat U^T)\xi\right)^{-\frac12},\label{eq:varphi_U_Onp2}\\
    \label{eq:psi_U_Onp2}
    \psi_{\hat U}(U) &= \hat U\log_m\left(\hat U^TU (U^T\hat U \hat U^TU)^{-\frac12}\right) + (I_p-\hat U\hat U^T)U(U^T\hat U \hat U^TU)^{-\frac12}.
\end{align}
\rev{The maximum dimension of matrices that need to be formed in practice is thus $n\times p$, and the computational complexity of evaluating $\varphi_{\hat U}, \psi_{\hat U}$ is $\mathcal{O}(np^2)$. Matrix functions (exponential, logarithm, square root, inversion,...) need only to be applied to $p\times p$-matrices.}
Since $A=\hat U^T\xi\in \Skew(p)$, computing the matrix exponential is efficient and stable.\footnote{Nevertheless, in practice, we will omit the matrix exponential by resorting to its Cayley approximation, which is second-order accurate.}
Rearranging terms, we get
\begin{equation}
\label{eq:varphi_U_Onp2_efficient}
  \varphi_{\hat U}(\xi) = 
  \left(\hat U(\exp_m(A) - A) + \xi\right)
  \left(I_p+\xi^T\xi + A^2\right)^{-\frac12}.
\end{equation}
With the SVD $MSR^T = \hat U^TU$, the argument of the matrix logarithm that appears in $\psi_{\hat U}(U)$ is the orthogonal matrix $\hat U^TU (U^T\hat U \hat U^TU)^{-\frac12}=MR^T$. The expression reduces to
\begin{equation}
\label{eq:psi_U_Onp2_efficient}
  \psi_{\hat U}(U) = \hat U(\log_m(MR^T)-MR^T) +  URS^{-1}R^T.
\end{equation}
It is interesting to observe that $MR^T\in O(p)$ is the solution to the Procrustes problem 
$\min_{Q\in O(p)}\|U - \hat UQ\|_F$, i.e., it is the rotation that brings $U$ closest to $\hat U$, see \cite[Section 6.4.1]{GolubVanLoan4th}.
\subsection{Retraction order}\label{sec:order}
Retractions are approximations of the Riemannian exponential map. The Riemannian exponential depends on the chosen metric.
Given a base point $\hat U\in \St(n,p)$ and an arbitrary but fixed orthogonal completion $\widehat{\mathbf{Q}}= [\hat U \ \hat U_\perp]\in O(n)$, the Riemannian exponential on the Stiefel manifold under the one-paramater family of metrics of \cite{HueperMarkinaLeite2020} reads
\begin{eqnarray}
\nonumber
    \Exp_{\hat U}:T_{\hat U}\St(n,p)&\to& \St(n,p), \\
\label{eq:St_exp}
    \xi = \widehat{\mathbf{Q}}\left[\begin{smallmatrix}
    A\\
    B
\end{smallmatrix}\right]
&\mapsto &
\widehat{\mathbf{Q}}\exp_\mathrm{m}\left(\begin{bmatrix}
        2\beta A&-B^\top\\
        B&0
    \end{bmatrix}\right)   
    \begin{bmatrix}
   I_p\\
    0
  \end{bmatrix}\exp_{\mathrm{m}}((1-2\beta)A),
\end{eqnarray}
see ~\cite[eq. (11)]{ZimmermannHueper2022}.
The canonical and the Euclidean metric correspond to $\beta=\frac12$ and $\beta = 1$, respectively.
By definition, a retraction of order $k$ coincides with the Taylor expansion of $\xi \mapsto \Exp_{\hat U}(\xi)$ around $0\in T_{\hat U}\St(n,p)$ up to terms of $k$'th order.
It holds that 
\begin{equation}\label{eq:Exp_Taylor}
    \xi \mapsto  \Exp_{\hat U}(0+\xi) = \Exp_{\hat U}(0) + D(\Exp_{\hat U})_0[\xi] + \frac12 D^2(\Exp_{\hat U})_0[\xi,\xi] + \mathcal{O}(\|\xi\|^3),
\end{equation}
where
\begin{align}\label{eq:Exp_derivatives}
    \Exp_{\hat U}(0) =
    \widehat{\mathbf{Q}}\left[\begin{smallmatrix}
    I\\
    0
\end{smallmatrix}\right]=U, \hspace{0.1cm}
    D(\Exp_{\hat U})_0[\xi]=\widehat{\mathbf{Q}}\left[\begin{smallmatrix}
    A\\
    B
\end{smallmatrix}\right] = \xi, \hspace{0.1cm}
 D^2(\Exp_{\hat U})_0[\xi,\xi]=
 \widehat{\mathbf{Q}}\left[\begin{smallmatrix}
    A^2-B^TB\\
    (2-2\beta)BA
\end{smallmatrix}\right]. 
\end{align}
The middle equation means $D(\Exp_{\hat U})_0=\mathrm{id}\big\vert_{T_U\St(n,p)}$, which is a classical fact from Riemannian geometry.
Since this is independent of the metric parameter, any retraction is a (first-order) retraction for the whole metric family.
It is well known that the classical polar factor retraction from \eqref{eq:classic_PFR} is of second-order under the Euclidean metric.\footnote{The Taylor expansion is
$
R_{\hat U}^{\textsf{PF}} (t\xi)
= \hat U
+ t\,\xi
+ \frac12 t^2\, \hat U (A^2 - B^T B)
+ O(t^3).
$}
The maps $\varphi_{\hat U}$ from \eqref{eq:varphi_U} share this property.
\begin{lemma}\label{lem:second_order_ret}
    The family of maps $\{\varphi_{\hat U}\mid \hat U\in \St(n,p)\}$ from \eqref{eq:novel_PFR},\eqref{eq:varphi_U} is a retraction under any metric of the one-parameter family.
    Under the Euclidean metric, it is a retraction of second order.
\end{lemma}
\begin{proof}
{\color{black} Using the series expression for the matrix exponential and those of matrix inversion, and the matrix square root,
\begin{equation*}
    (I+X)^{\frac{1}{2}}=\sum_{i=0}^\infty \frac{(-1)^i(2i)!}{(1-2i)(i!)^2 4^i}X^i = I+\frac{t}{2}X+\mathcal{O}(\|X\|^2), \ \ (I-X)^{-1}=\sum_{i=0}^\infty X^i,
\end{equation*}}
a series expansion of $t\mapsto \varphi_{\hat U}(t\xi)$ at $t=0$ in the direction $\xi = \widehat{\mathbf{Q}}\left[\begin{smallmatrix}
    A\\
    B
\end{smallmatrix}\right]\in T_{\hat U}\St(n,p)$ is seen to be
{\color{black}\begin{align*}
    \varphi_{\hat U}(t\xi)&=\widehat{\mathbf{Q}}\varphi_E(0+t
    \left[\begin{smallmatrix}
        A\\
        B
    \end{smallmatrix}\right])
    = \widehat{\mathbf{Q}}\begin{bmatrix}
		\exp_m(tA)\\
		tB
	\end{bmatrix}(I_{p} + t^2B^TB)^{-\frac12}\\
    &=\widehat{\mathbf{Q}}\begin{bmatrix}
        I+tA+\frac{1}{2}t^2A^2+\mathcal{O}(t^3)\\
        tB
    \end{bmatrix}(I-\tfrac{t^2}{2}B^TB+\mathcal{O}(t^4))\\
    &=\widehat{\mathbf{Q}}
\begin{bmatrix}
I + tA + \frac12 t^2(A^2-B^T B) + \mathcal{O}(t^3)\\
tB - \frac12 t^3 BB^TB + \mathcal{O}(t^5)
\end{bmatrix}, \ \text{ for } t\to 0.
\end{align*}}
Hence, $D(\varphi_{\hat U})_0[\xi]= \frac{d}{dt}\vert_{t=0} \varphi_{\hat{U}}(t\xi)=\widehat{\mathbf{Q}}\left[\begin{smallmatrix}
    A\\
    B
\end{smallmatrix}\right]$
and
$
D^2(\varphi_{\hat U})_0[\xi,\xi]=\frac{d^2}{dt^2}\vert_{t=0} \varphi_{\hat{U}}(t\xi)=
 \widehat{\mathbf{Q}}\left[\begin{smallmatrix}
    A^2-B^TB\\
    0
\end{smallmatrix}\right]
$
so that 
$\varphi_{\hat U}$ matches the Riemann exponential $\Exp_{\hat U}$ under any $\beta$-metric  up to terms of order one, and up to terms of order two under the Euclidean metric ($\beta=1$), \rev{cf. \eqref{eq:Exp_derivatives}}.\\
Equation \eqref{eq:varphi_U_Onp2} reveals smooth, even analytic dependence on the base point $\hat U\in \St(n,p)$.
\end{proof}
For ease of notation, we will from now on denote the polar--light retraction \eqref{eq:varphi_U} by $R^{\textsf{PL}}$.
\paragraph{Two special cases}
Stiefel tangent vectors come in the form
$
\xi =  \hat U A + \hat U_\bot B = 
(\hat U\hat U^T)\xi + (I-\hat U\hat U^T)\xi
$ with a (skew-symmetric) component in $\text{span}(\hat U)$ and a component in  $\text{span}(\hat U)^\bot$.
In this sense, we write
\[
    T_{\hat U} \St(n,p) =
    T^{\parallel}_{\hat U} \St(n,p)\oplus T^{\bot}_{\hat U} \St(n,p). 
\]
\begin{lemma}
\label{lem:light_between_PF_RL}
When both maps are restricted to $T^{\parallel}_{\hat U} \St(n,p)$, the polar-light retraction $R^{\textsf{PL}}$ coincides with the Riemannian exponential
\[
    R^{\textnormal{\textsf{PL}}}_{\hat U}\big\vert_{T^{\parallel}_{\hat U} \St(n,p)}
    = \Exp_{\hat U}\big\vert_{T^{\parallel}_{\hat U} \St(n,p)}, \quad
    (R^{\textnormal{\textsf{PL}}}_{\hat U})^{-1}\big\vert_{ R^{\textnormal{\textsf{PL}}}_{\hat U}\left(T^{\parallel}_{\hat U} \St(n,p)\right)}
    = \Log_{\hat U}\big\vert_{\Exp_{\hat U}\left(T^{\parallel}_{\hat U} \St(n,p)\right)}.
\]
When restricted to $T^{\bot}_{\hat U} \St(n,p)$, $R^{\textnormal{\textsf{PL}}}$ coincides with the full polar factor retraction \eqref{eq:classic_PFR}
\[
    R^{\textnormal{\textsf{PL}}}_{\hat U}\big\vert_{T^{\bot}_{\hat U} \St(n,p)}
    = R^{\textnormal{\textsf{PF}}}_{\hat U}\big\vert_{T^{\bot}_{\hat U} \St(n,p)},\quad
    R^{\textnormal{\textsf{PL}}}_{\hat U}\big\vert_{ R^{\textnormal{\textsf{PL}}}_{\hat U}\left(T^{\bot}_{\hat U} \St(n,p)\right)}
    = (R^{\textnormal{\textsf{PF}}}_{\hat U} )^{-1}\big\vert_{R^{\textnormal{\textsf{PF}}}_{\hat U}\left(T^{\bot}_{\hat U} \St(n,p)\right)}.
\]
\end{lemma}
\begin{proof}
This is readily seen by comparing the equations \eqref{eq:classic_PFR}, \eqref{eq:varphi_U_Onp2}, \eqref{eq:psi_U_Onp2} and equation \eqref{eq:St_exp} in the above special cases, where either $A=0$ or $B=0$.
Note that when restricted to either $T^{\parallel}_{\hat U} \St(n,p)$ or to
$T^{\bot}_{\hat U} \St(n,p)$,
the Riemannian exponential is indepedent of the $\beta$-metric parameter.
\end{proof}
\rev{
\begin{remark*}
Tangent vectors from $T^{\bot}_{\hat U} \St(n,p)$, i.e., tangent vectors of the form $\xi = (I-\hat U \hat U^T)\xi$ correspond to tangent vectors on the Grassmann manifold. For this type of input,  the classical polar factor retraction \eqref{eq:classic_PFR} and the polar-light retraction \eqref{eq:varphi_U_Onp2} both read
\[
    R_{\hat U}(\xi) 
    = ({\hat U}+\xi)\left(I_p + \xi^T\xi\right)^{-\frac12}\in \St(n,p).
\]
Any Stiefel matrix ${\hat U}$ defines a projector ${\hat U}{\hat U}^T$, which is an element of the Grassmann manifold.
Hence, both of the Stiefel polar-factor retractions yield the same Grassmann retraction via
\[
   R_{\hat U}(\xi)R_{\hat U}(\xi)^T = 
   ({\hat U}+\xi)\left(I_p + \xi^T\xi\right)^{-1}({\hat U}+\xi)^T.
\]
This Grassmann retraction is investitaged in our paper \cite{jensen2025maxvol}.
Section 3.5 in this refrence explains the relation between the Grassmann and the Stiefel polar factor retractions.
\end{remark*}
}
By the above lemma, we expect that for tangent vectors $
\xi =  \hat U A + \hat U_\bot B
$ with dominant $A$-component, the polar-light retraction will be closer to the Riemannian exponential than the full polar factor retraction.
If the $B$-component dominates, we expect that the full polar and the polar-light retraction produce similar results.
This is confirmed by the numerical experiments in the next section.

\subsection{Computing the polar--light retraction and its inverse}\label{sec:eff_ret_comp} 
\rev{The formulas \eqref{eq:varphi_U_Onp2_efficient} and \eqref{eq:psi_U_Onp2_efficient} feature the matrix exponential and the matrix logarithm, respectively, which can be computationally demanding.}
The Cayley transformations
provide structure-preserving second-order approximations
\[
\exp_m(A)\approx \Cay(\tfrac{1}{2}A)= (I-\tfrac12 A)^{-1}(I+\tfrac12A),
\hspace{0.15cm} \log_m(R)\approx \Cay^{-1}(R)= 2(R+I)^{-1}(R-I).
\]
Here, structure-preserving means that 
\[
    \Cay\left(\Skew(p)\right) \subset SO(p),\quad \Cay^{-1}\left(SO(p)\right)\subset \Skew(p).
\]
As with the matrix logarithm, $\Cay^{-1}$ is only well-defined for matrices that do not feature $-1$ as an eigenvalue.
As with the matrix exponential, $\Cay$ is well-defined for all skew-symmetric matrices: The eigenvalues of a skew-symmetric matrix are imaginary so that the matrix factor $(I-\frac12 A)^{-1}$ cannot be singular.

With all input data as introduced for \eqref{eq:varphi_U_Onp2_efficient}, \eqref{eq:psi_U_Onp2_efficient}, the maps
\begin{eqnarray}
\label{eq:varphi_U_Onp2_cay}
  R^{\textsf{PL Cay}}_{\hat U}(\xi) &=& 
  \left(\hat U(\Cay(\tfrac{1}{2}A) - A) + \xi\right)
  \left(I_p+\xi^T\xi + A^2\right)^{-\frac12},\\
\label{eq:psi_U_Onp2_cay}
  (R^{\textsf{PL Cay}})^{-1}_{\hat U}(U) &=& \hat U(\Cay^{-1}(MR^T)-MR^T) +  URS^{-1}R^T
\end{eqnarray}
form a pair of retraction/inverse retraction. {\color{black} Since the Cayley transformation is a second-order approximation of the matrix exponential, the same calculations as in \Cref{lem:second_order_ret} shows that \eqref{eq:varphi_U_Onp2_cay} is a second-order retraction under the Euclidean metric.}

Computing the real Schur form of the skew--symmetric matrix $A$,  $A=WTW^T$, we obtain the block--diagonal matrix $T$, featuring either $(2\times 2)$ blocks or $(1\times 1)$ zero blocks on the diagonal. Thus, we can efficiently evaluate the matrix exponential of matrices of this type, as well as compute $\left(I-\frac{1}{2}T\right)^{-1}$, which feature in the Cayley transformation $\Cay(\tfrac{1}{2}A)=W\Cay(\tfrac{1}{2}T)W^T$. 

Given a fixed $\xi\in T_{\hat U}\St(n,p)$, consider the case that we want to evaluate the polar--light retraction $R^{\textsf{PL}}_{\hat U}(t\xi$), for various $t$. For this to be efficient, we can compute a QR decomposition of $(I_p+UU^T)\xi=QR$, and obtain $I_p+t^2\xi^T(I_p+UU^T)\xi=I_p+t^2R^TR$, so that one only needs to form $R$. One proceeds by computing the SVD $R=\Gamma \Sigma V^T$, so that $I_p+t^2R^TR=V(I_p+t^2\Sigma^2)V^T$, for which the matrix square--root and --inverse can be efficiently computed by applying the operations on the diagonal elements of $(I_p+t^2\Sigma^2)$ only, as it holds that $(V(I_p+t^2\Sigma^2)V^T)^{-\frac12}=V(I_p+t^2\Sigma^2)^{-\frac12}V^T$. For the exponential term, compute the real Schur form $A=WTW^T$ and obtain $R^{\textsf{PL}}_{\hat U}(t\xi)=((\hat UW)(\exp_m(tT)-tT)+t(\xi W))(W^TV)(I_p+t^2\Sigma^2)^{-\frac{1}{2}}V^T$. One can replace $\exp_m$ with $\Cay$ for a Cayley--accelerated variant. For both variants, one can exploit the block--structure of $tT$ when computing the matrix exponential $\exp_m(tT)$ \cite[pp. 43]{gallier2020differential}, or the inverse $\left(I-\frac{t}{2}T\right)^{-1}$.

\subsection{The classical Stiefel Cayley retraction}
\label{sec:cayley_comp}
For comparison purpose, we recall the classical Cayley retraction \cite{Wen:2013,Zhu2020} on the Stiefel manifold
\begin{equation}\label{eq:full_cayley}
    R^{\textsf{Cay}}_{\hat U}(t\xi)=\Cay\left(\tfrac{t}{2}(P_{\hat U}\xi \hat U^T-\hat U\xi^TP_{\hat U})\right)\hat U.
\end{equation}
with $P_{\hat U}=I-\frac{1}{2}\hat U \hat U^T$. While the Cayley transformation features in both $R^{\textsf{Cay}}$ and $R^{\textsf{PL Cay}}$, the two retractions are inherently different. 
The economy-size equivalent to \eqref{eq:full_cayley}
is
\begin{equation}\label{eq:Cayley_retraction_p}
    R^{\textsf{Cay}}_{\hat U}(t\xi)=-\hat U+(t(\xi-\hat UA)+2\hat U)\left(\tfrac{t^2}{4}B^TB-\tfrac{t}{2}A+I\right)^{-1},
\end{equation}
analogous to \cite[Proposition 5.2]{BendokatZimmermann:2021}, \cite[Proposition 4.5]{Tiep:2025}. This is the form that we use in the numerical experiments.
Note that computing \cref{eq:Cayley_retraction_p} only requires inversion of a ($p\times p$) matrix.

The Cayley retraction $R^{\textsf{Cay}}$ is second--order accurate under the canonical metric.
We expect this to be known, but we could not find a literature reference. 
The differential terms of first and second order are
\[
D(R^{\textsf{Cay}}_{\hat U})_0[\xi]=\widehat{\mathbf{Q}}\left[\begin{smallmatrix}
    A\\
    B
\end{smallmatrix}\right],
\text{ and }
D^2(R^{\textsf{Cay}}_{\hat U})_0[\xi,\xi]=
 \widehat{\mathbf{Q}}\left[\begin{smallmatrix}
    A^2-B^TB\\
    BA
\end{smallmatrix}\right]
\text{ with } \xi = \widehat{\mathbf{Q}}\left[\begin{smallmatrix}
    A\\
    B
\end{smallmatrix}\right].
\]
A comparison with the Taylor series \eqref{eq:Exp_Taylor},\eqref{eq:Exp_derivatives} of the Exponential map shows second-order consistency for $\beta = \frac{1}{2}$ (canonical metric) and first-order consistency for $\beta = 1$ (Euclidean metric).

The inverse to \eqref{eq:Cayley_retraction_p} is 
\begin{equation*}
    (R_{\hat U}^{\textsf{Cay}})^{-1}(V)=\xi=2\hat UF^T+2VF-2\hat U, \ \ F=(I_p+\hat U^TV)^{-1}, 
\end{equation*}
see \cite[Equation 54]{Zhu2020}.

%
\section{Experimental results}
\label{sec:experiments}
In this section we present five numerical experiments investigating the properties of the new polar-light retraction. We also compare it to other existing alternatives. The source code is publicly available.\footnote{\url{https://github.com/JensenRasmus/PolarLightStiefel}}
\subsection{Accuracy of the PL retraction relative to the (Eculidean) Riemann exponential}

First, we evaluate the accuracy of the polar-light retraction by quantifying its deviation from the corresponding Riemannian geodesic.
To this end, we compute a pseudo-random data triple
$U_0,U_1\in \St(n,p)$, $\xi\in T_{U_0}\St(n,p)$ such that 
\[
    \Exp_{U_0}(\xi) = U_1,\quad \dist_{E}(U_0,U_1)= \frac{\pi}{2},
\]
where $\dist_E$ denotes the Euclidean Stiefel metric associated with $\beta = 1$ in \eqref{eq:St_exp}.
The geodesic connecting $U_0,U_1$ is 
\[
    \gamma_E: [0,1]\to \St(n,p), t \mapsto \Exp_{U_0}(t\xi).
\]
To obtain a retraction $R_{U_0}$ connecting the same endpoints, we compute $R_{U_0}^{-1}(U_1)$. This yields a tangent vector $\xi_R$ with $R_{U_0}(\xi_R) = U_1$.
The retraction curve connecting the given endpoints is
\[
    \gamma_R: [0,1]\to \St(n,p), t \mapsto R_{U_0}(t\xi_R). 
\]
We discretize the unit interval $[0,1]$ in $51$ equidistant steps $t_k$ and compute the error between the geodesic and the retraction curve $\|\gamma_E(t_k) - \gamma_R(t_k)\|_F$. As retractions, we consider the polar factor retraction $R^{\textsf{PF}}$ \eqref{eq:classic_PFR} and the polar--light retraction $R^{\textsf{PL}}$ \eqref{eq:varphi_U_Onp2_efficient}.

\Cref{fig:testfig} displays the error curves for dimensions $n=1000$ and \\
$p\in\{400,200,100,50\}$. The maximal errors are listed in \Cref{tab:errors}. 

\begin{figure}[t!]
    \centering
    \begin{subfigure}[t]{0.45\textwidth}
        \centering
        \includegraphics[width=1.0\textwidth]{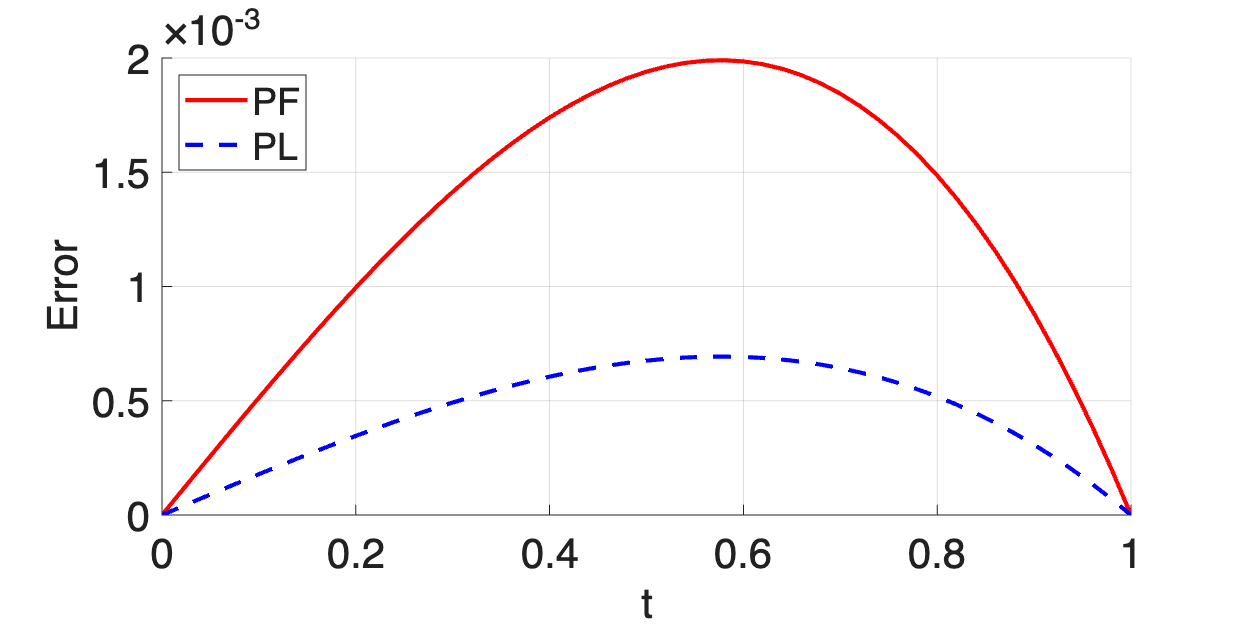}
        \caption{for $U_0,U_1\in \St(n,p)$, $n=1000, p=400$, dist$(U_0,U_1)=\frac{\pi}{2}$.}
    \end{subfigure}%
    ~ 
    \begin{subfigure}[t]{0.45\textwidth}
        \centering
        \includegraphics[width=1.0\textwidth]{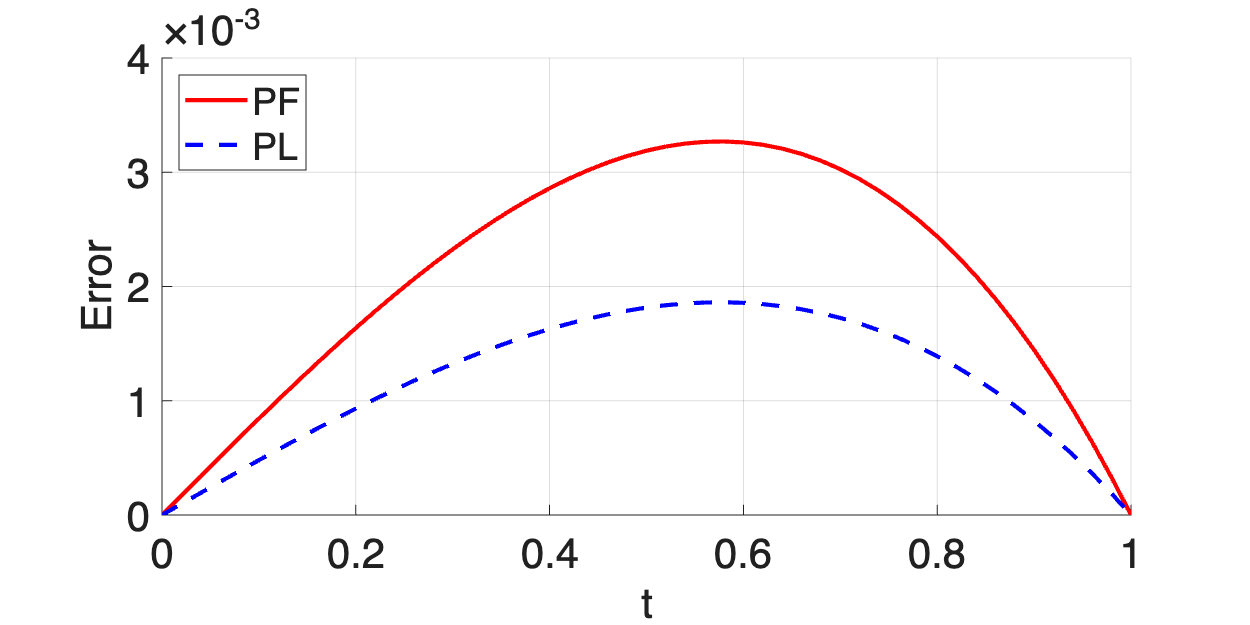}
        \caption{for $U_0,U_1\in \St(n,p)$, $n=1000, p=200$, dist$(U_0,U_1)=\frac{\pi}{2}$.}
    \end{subfigure}
     \begin{subfigure}[t]{0.45\textwidth}
        \centering
        \includegraphics[width=1.0\textwidth]{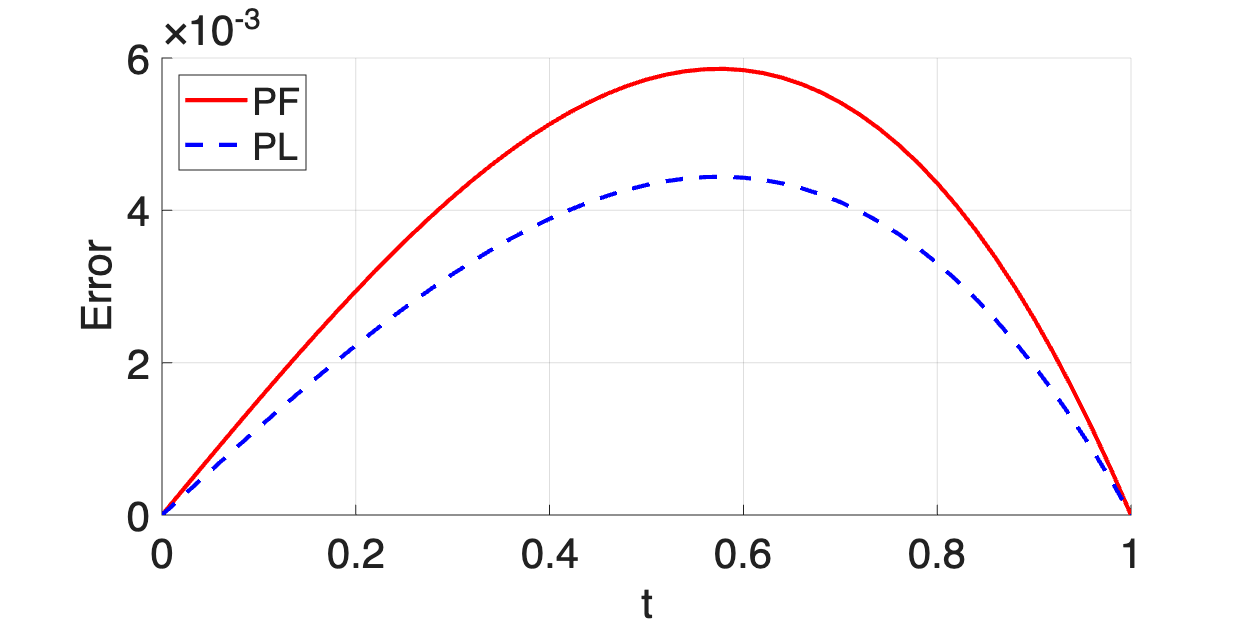}
        \caption{for $U_0,U_1\in \St(n,p)$, $n=1000, p=100$, dist$(U_0,U_1)=\frac{\pi}{2}$}
    \end{subfigure}%
    ~ 
    \begin{subfigure}[t]{0.45\textwidth}
        \centering
        \includegraphics[width=1.0\textwidth]{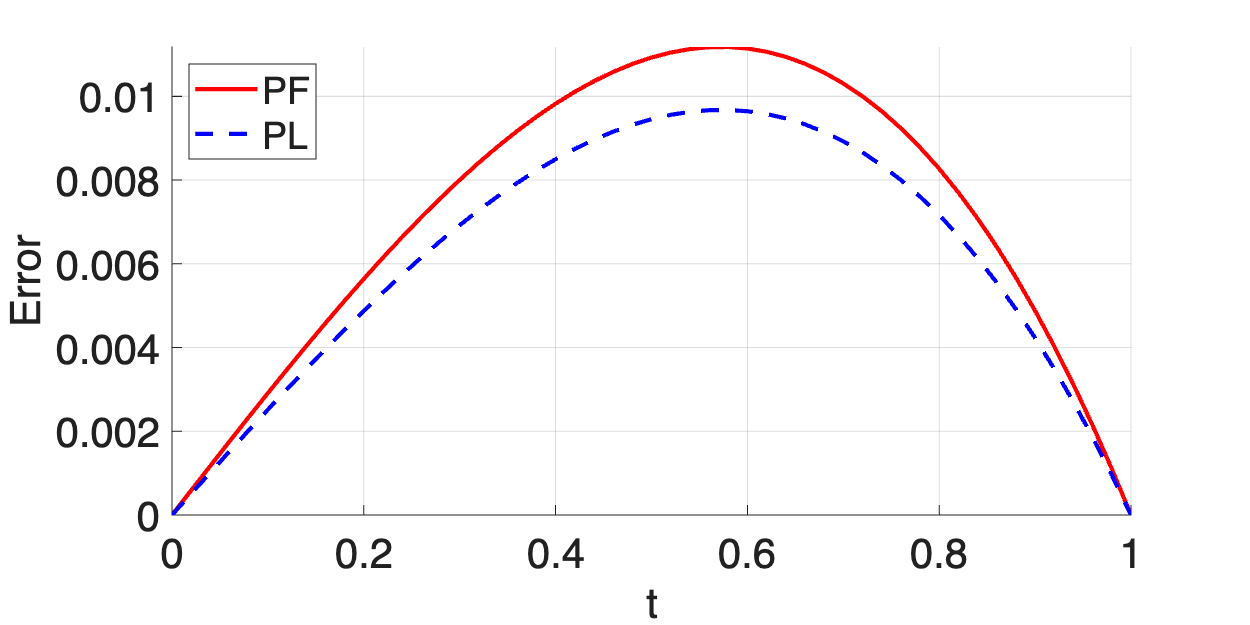}
        \caption{for $U_0,U_1\in \St(n,p)$, $n=1000, p=50$, dist$(U_0,U_1)=\frac{\pi}{2}$}
    \end{subfigure}   
    \caption{The figure shows the error $\|\Exp_{U0}(t\xi) - R_{U_0}(t\xi_R)\|_{F}$ between the Riemannian geodesic connecting $U_0$ and $U_1 = \Exp_{U_0}(t\xi)\big\vert_{t=1}$ and a retraction curve from $U_0$ to the same point $U_1 = R_{U_0}(t \xi_R)\big\vert_{t=1}$.
    As retractions, the polar factor retraction $R^{\textsf{PF}}$ (PF, solid red line) and the 
    polar-light retraction $R^{\textsf{PL}}$ (PL, dashed blue line) are considered.}
    \label{fig:testfig}
\end{figure}

\begin{table}[htbp]
\footnotesize
\begin{center}
  \begin{tabular}{|l|c c|} \hline
   $p$ &  max error PF &  max error PL \\ \hline
    400 & 1.99e-3 & 6.93e-4 \\
    200 & 3.27e-3 & 1.86e-3 \\ 
    100 & 5.86e-3 & 4.45e-3 \\ 
    50  & 1.12e-2 & 9.67e-3 \\ \hline
  \end{tabular}
\end{center}
\caption{Error maxima of the retraction curves considered in \Cref{fig:testfig}}\label{tab:errors}
\end{table}
%
%
As can be seen from the figure and the table,
the polar-light retraction is consistently closer to the Riemannian geodesic than the full polar factor retraction. The difference is more pronounced for larger values of $p$. Intuitively, this makes sense, since the difference between the PF- and the PL-retraction is only in the treatment of the $(p\times p)$-block $A$, cf. \eqref{eq:classic_PFR} and \eqref{eq:varphi_U_Onp2_efficient}.
As shown in \Cref{lem:light_between_PF_RL}, if only the $A$-component were present, the polar-light retraction would become the Riemannian exponential, while for $A=0$ it coincides with the full polar factor retraction.
%
%
%
%
%
%
%

\subsection{Accuracy of the inverse polar--light retraction relative to the Riemannian logarithm}
The next experiment addresses how close the inverse retraction $(R^{\textsf{PL}})^{-1}$ is to the Riemannian logarithm (RL) under the Euclidean metric.
Again, we start with a pseudo-random data triple
$U_0,U_1\in \St(n,p)$, $\xi\in T_{U_0}\St(n,p)$ such that 
\[
    \Exp_{U_0}(\xi) = U_1,\quad \dist_{E}(U_0,U_1)= \frac{\pi}{2},
\]
and compute the geodesic
$\gamma_E: [0,1]\to \St(n,p), t \mapsto \Exp_{U_0}(t\xi).$

The geodesic is a manifold curve that we map back to the coordinate domain, in this case the tangent space, via the inverse of a retraction $R_{U_0}$.
To assess how close the inverse retraction under consideration is to the Riemannian logarithm, we compute the error
\[
    \|\Log_{U_0} (\gamma_E(t)) - R_{U_0}^{-1}(\gamma_E(t)\|_F.
\]
In theory, when the Riemannian logarithm is used as an inverse retraction,
the tangent space curve is $t\mapsto t\xi$.
However, the Riemannian logarithm is not available in closed form but has to be computed by an iterative procedure \cite{ZimmermannHueper2022, MataigneStiefelLog:2025}. Therefore,
we also assess the accuracy
\[
    \|\Log_{U_0} (\gamma_E(t)) - t\xi\|_F.
\]
As inverse retractions, we use the inverses of the full polar factor retraction (PF) \cite[Algorithm 1]{Kaneko:2013} and the inverse of the polar-light retraction (PL) \eqref{eq:psi_U_Onp2_efficient}.
\Cref{fig:testfig2} (a) and (b) displays the error curves for dimensions $n=1000$ and $p\in\{400,50\}$.

We repeat the experiment in the exact same set-up, but replace the inverse polar-light retraction with its Cayley accelerated variant (PL Cay) of \eqref{eq:psi_U_Onp2_cay}, The results are presented in \Cref{fig:testfig2} (c) and (d). To the naked eye, the results are indifferent.
\begin{figure}[t!]
    \centering
    \begin{subfigure}[t]{0.45\textwidth}
        \centering
        \includegraphics[width=1.0\textwidth]{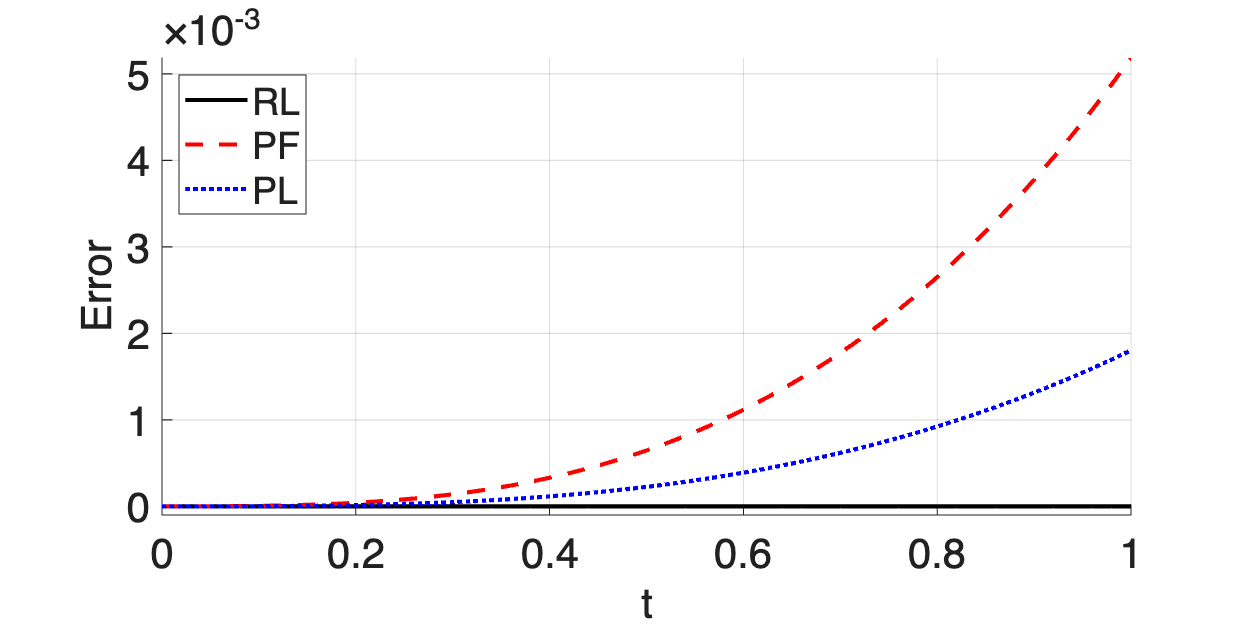}
        \caption{for $U_0,U_1\in \St(n,p)$, $n=1000, p=400$, dist$(U_0,U_1)=\frac{\pi}{2}$.}
    \end{subfigure}%
    ~ 
    \begin{subfigure}[t]{0.45\textwidth}
        \centering
        \includegraphics[width=1.0\textwidth]{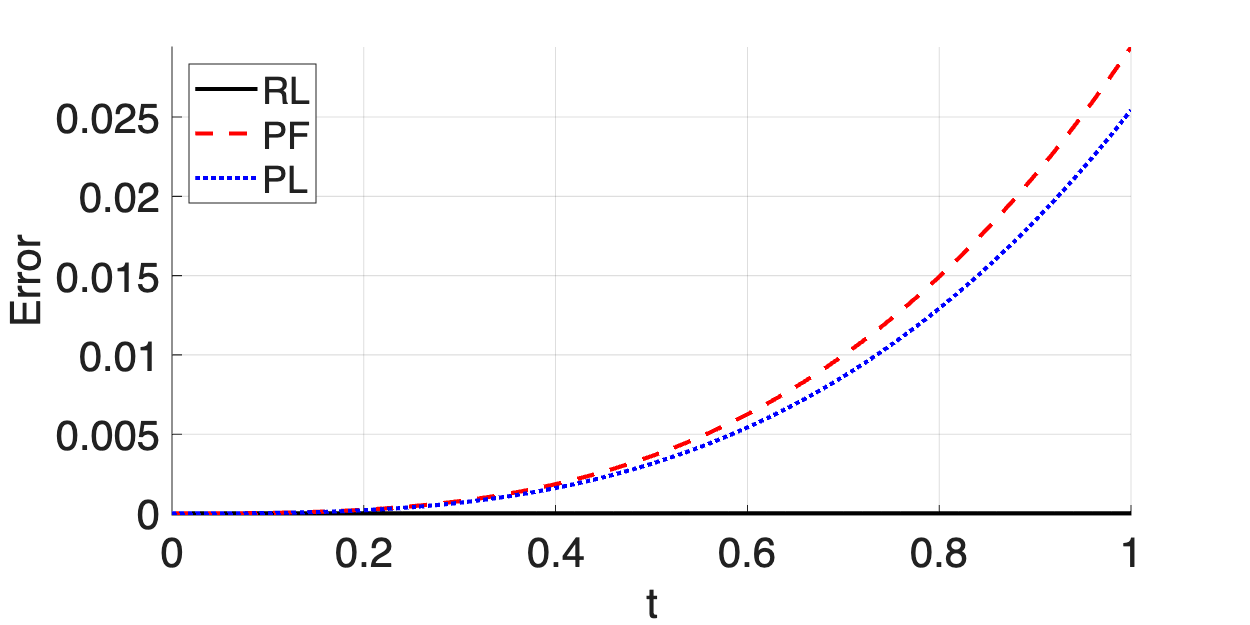}
        \caption{for $U_0,U_1\in \St(n,p)$, $n=1000, p=50$, dist$(U_0,U_1)=\frac{\pi}{2}$}
    \end{subfigure}  
    -
    \begin{subfigure}[t]{0.45\textwidth}
        \centering
        \includegraphics[width=1.0\textwidth]{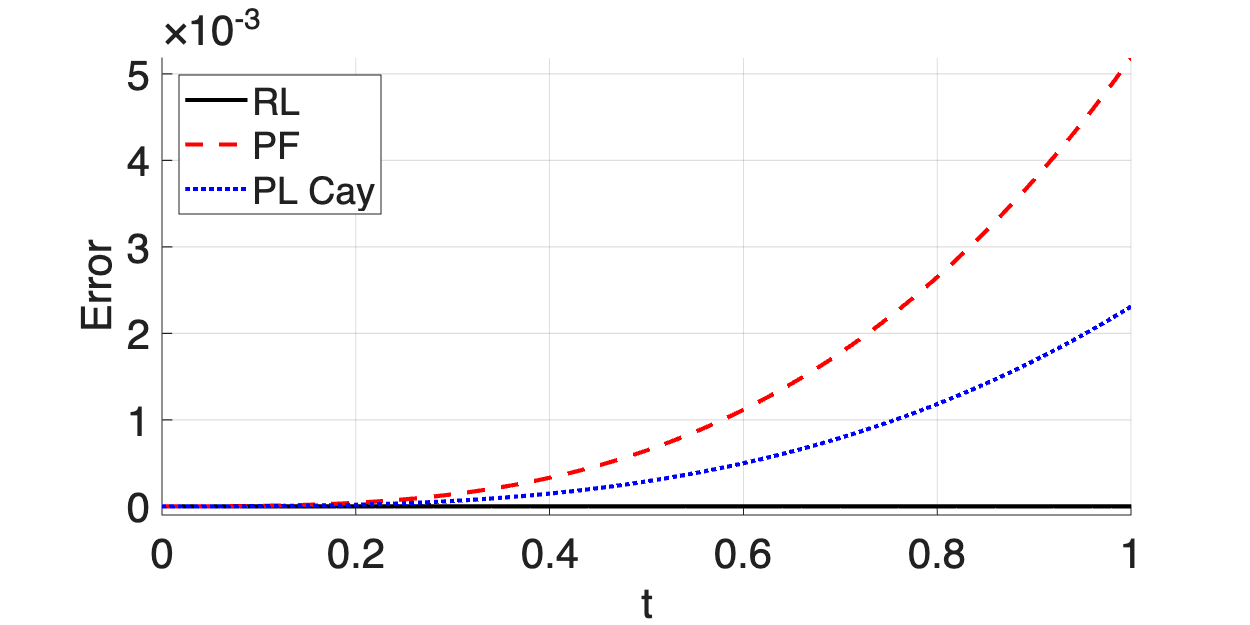}
        \caption{for $U_0,U_1\in \St(n,p)$, $n=1000, p=400$, dist$(U_0,U_1)=\frac{\pi}{2}$.}
    \end{subfigure}%
    ~ 
    \begin{subfigure}[t]{0.45\textwidth}
        \centering
        \includegraphics[width=1.0\textwidth]{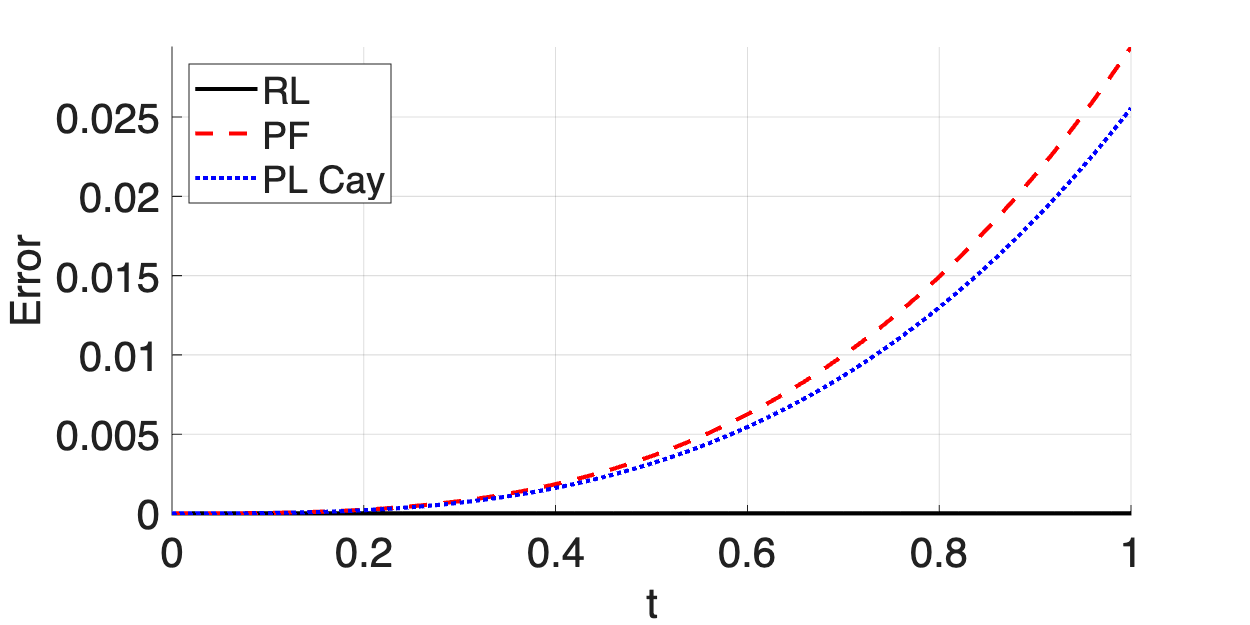}
        \caption{for $U_0,U_1\in \St(n,p)$, $n=1000, p=50$, dist$(U_0,U_1)=\frac{\pi}{2}$}
    \end{subfigure}  
    \caption{The plots (a) and (b) shows the errors $\|\Log_{U0}(\gamma_E(t)) - R_{U_0}^{-1}(\gamma_E(t))\|_{F}$ between the 
    coordinate curves obtained from pulling back a Riemannian geodesic $t\to \gamma_E$
    to the tangent space with the Riemannian logarithm and 
    the inverse polar factor retraction \eqref{eq:classic_PFR} (PF, dashed red line) and the inverse
    polar-light retraction \eqref{eq:varphi_U_Onp2_efficient} (PL, dotted blue line) .
    By construction, the coordinate curves all start from $0 = \Log_{U0}(\gamma_E(0))= \Log_{U_0}(U_0) = R_{U_0}^{-1}(0)$ in the tangent space.
    The baseline error is $\|\Log_{U0}(\gamma_E(t)) - t\xi\|_{F}$, which theoretically is zero, but practically illustrates the numerical accuracy of the Riemannian logarithm (RL, solid black line). In the plots (c) and (d) the matrix logarithm is replaced with the inverse Cayley transform in the inverse polar-light retraction (PL Cay, dotted blue line). The results are virtually the same}
    \label{fig:testfig2}
\end{figure}

The maximal errors are listed in \Cref{tab:errors2}. 

\begin{table}[htbp]
\footnotesize
\begin{center}
  \begin{tabular}{|l|c c c c|} \hline
   $p$  & max error RL (ref) & max error PF &  max error PL & max error PL Cay\\ \hline
    400 & 2.39e-12          & 5.18e-3     & 1.80e-3      & 2.31e-3\\
    50  & 2.50e-13          & 2.94e-2     & 2.54e-2      & 2.56e-2\\ \hline
  \end{tabular}
\end{center}
\caption{Error maxima of the retraction curves considered in Figure \ref{fig:testfig2}}
\label{tab:errors2}

\end{table}

%

%
%
%
%
%
%
%

\subsection{Computation time}\label{sec:experiments_timing}

To assess the computational effort, we create $100$ pseudo-random points $U_1\in\St(n,p)$ of dimensions $n=1000,p=400$, and we fix a $U_0\in\St(n,p)$.
Then, for each retraction $R_{U_0}$ included in the competition,
we measure the computation time\footnote{Timing results are obtained with Python 3 on a MacBook Air M2 with 16GB RAM
} of performing $100$ calculations
of $R_{U_0}^{-1}(U_1)$.
On the side, we also assess the identity $R_{U_0}^{-1} \circ R_{U_0}=\id_{T_{U_0}\St(n,p)}$ by computing the norm
$\|R_{U_0}^{-1} \circ R_{U_0}(\xi)-\xi\|_F$ averaged over the number of runs. In \Cref{tab:timings}  we report the average runtime and residual norm for the full polar factor retraction (PF) of \eqref{eq:classic_PFR}, the polar--light retraction (PL) of \eqref{eq:varphi_U_Onp2_efficient},  the PL retraction with $\exp_m$ replaced with $\Cay$ (PL Cay) as in 
\eqref{eq:varphi_U_Onp2_cay}, and the Cayley retraction (Cayley) of \eqref{eq:Cayley_retraction_p}. Moreover, for reference, we consider the Grassmann--like quasi geodesic (QD. Gr--like) \cite[Equation 3]{Bendokat_quasiGeo2021} and the QR retraction (QR) \cite[Equation 4.8]{AbsilMahonySepulchre2008}.

By comparing \eqref{eq:classic_PFR} to \eqref{eq:varphi_U_Onp2_efficient}, it is clear that the (forward) polar-light retraction is more costly to evaluate than the full polar factor retraction.
Apart from computations that do not depend on $n$, the essential difference is that the former features a term $(UM+ \xi)$, while for the latter, the corresponding term is only $(U+\xi)$. 
The table shows that all retraction maps feature an acceptable round-off error with respect to the structure identity $R_{U_0}^{-1} \circ R_{U_0}=\id_{T_{U_0}}$.
For the dimensions of $n=1000, p=400$, the inverse Cayley retraction is the most efficient one to evaluate on a single-query basis. It is about five times faster than the inverse polar-light retraction, which ranks second, and is, in turn, about five times faster than the next best competitor.
For the dimensions of $n=10000, p=1000$, the ranking is the same, but the speed-up factor between Cayley and PL Cay is reduced to $1.7$.

%
\begin{table}[htbp]
    \centering
    \footnotesize
    \begin{tabular}{l  c c  c c  }\hline
    & \multicolumn{2}{c}{$n=1000,p=400$}& \multicolumn{2}{c}{$n=10000,p=1000$}\\ 
    Retraction $R_U$ & avg. time (s.)& avg. error &avg. time (s.)& avg. error \\\hline
    inv. PF             & 0.099 & 2.13e-13 & 1.374 & 5.61e-13 \\
    inv. PL             & 0.093 & 1.54e-13 & 1.066 & 7.60e-13\\
    inv. PL Cay.     & 0.021 & 7.76e-14 & 0.354 & 1.88e-13\\
    inv. QD. Gr--like    & 0.116 & 1.54e-13 & 1.922 & 7.68e-13\\
    inv. QR             & 0.298 & 4.41e-14 & 4.621 & 1.19e-13\\
    inv. Cayley     & 0.004 & 5.27e-14 & 0.210 & 1.51e-13\\ \hline
    \end{tabular}
    \caption{Average computation time based on 100 inverse retraction evaluations on $\St(n,p)$. The error is computed as $\|R_{U_0}^{-1} \circ R_{U_0}(\xi)-\xi\|_F$, where $R_{U_0}$ is the selected retraction.}
    \label{tab:timings}
\end{table}
%
%
%
%
To further measure the computation time, we compute a pseudo-random data triple
$U_0,U_1\in \St(n,p)$, $\xi\in T_{U_0}\St(n,p)$ such that 
\[
    \Exp_{U_0}(\xi) = U_1,\quad \dist_{e}(U_0,U_1)= \pi,
\]
where $\dist_e$ denotes the Euclidean Stiefel metric associated with $\beta = 1$ in \eqref{eq:St_exp}.
We consider dimensions of $n=10.000$ and $p\in\{500,1000,1500,2000\}$.
In each case, we map the data `there and back', i.e., we compute
\[
  \tilde U_0=R_{U_0}(\xi),\quad  \tilde \xi = R_{U_0}^{-1}(\tilde U_0) \text{ for } R\in\{R^{\textsf{PF}}, R^{\textsf{PL}},R^{\textsf{PL Cay}},R^{\textsf{Cay}}\}.
\]
We average the wallclock times over 10 random runs.
\begin{figure}[t!]
    \centering
    \includegraphics[width=1.0\textwidth]{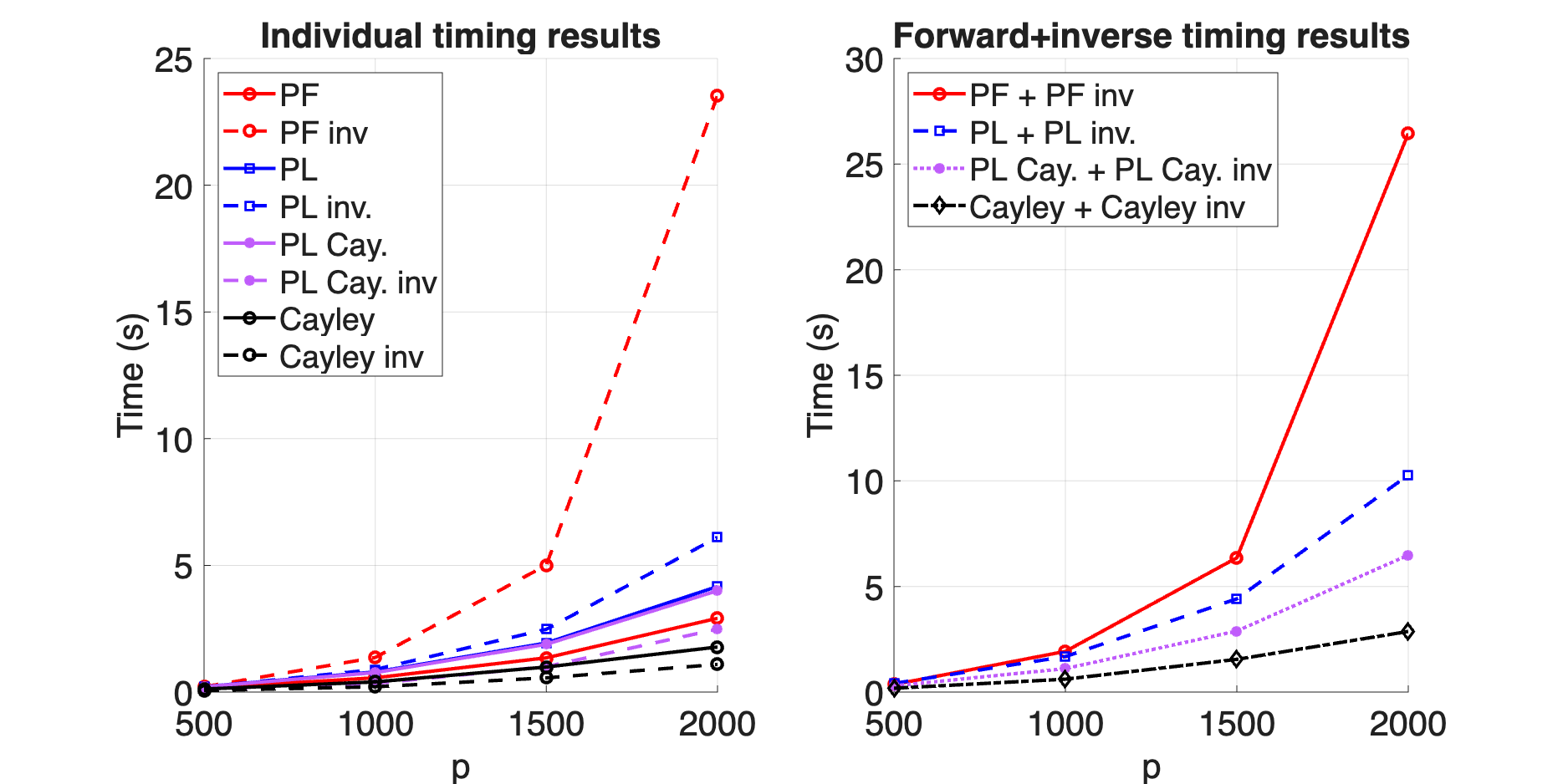}
    \caption{Timing results for computing the various retractions and their inverses for increasing dimension $p$. 
    }
    \label{fig:timefig}
\end{figure}
\Cref{fig:timefig} displays the results.
On the left, we observe that $R^{\textsf{Cay}}$ is the most efficient to evaluate, and the same holds for its inverse.
$R^{\textsf{PL}}$ and its inverse show favorable growth both in forward and inverse mode, when compared to the polar factor retraction. 
This is confirmed by the right subplot of \Cref{fig:timefig}, which shows the added costs of going `there and back' with the retractions under consideration.
At a dimension of $p=2000$, the $R^{\textsf{Cay}}$ is more than eight times faster than $R^{\textsf{PF}}$. $R^{\textsf{PL Cay}}$ is, by comparison, more than three times faster than $R^{\textsf{PF}}$. We observe that computing $(R^{\textsf{PL}})^{-1}$ is slower to compute than $R^{\textsf{PL}}$. We find that this is due to the expensive matrix logarithm featuring in $(R^{\textsf{PL}})^{-1}$.

As a third experiment, we simulate a multi-queries scenario. Given data points $U_0,U_1$, we consider the computational costs of obtaining a tangent vector \\
$\xi\in T_U\St(n,p)$ so that $U_1=R_{U_0}(\xi)$, and the subsequent costs of evaluating $\gamma_R (t_k)=R_{U_0}(t_k\xi)$, where $t_k=\frac{k}{N-1}$ for $k=0,\dots,N-1$. This mimics the procedures that are required for quasi-linear manifold interpolation.
In this setting, it is beneficial to implement $t$--dependent formulas, so that one reuses matrix--matrix products and matrix decompositions, which allows for efficient probing of the curve $t\mapsto\gamma_R (t)$.
For the $R^{\textsf{PL}}$ and the Cayley--accelerated variant $R^{\textsf{PL Cay}}$, this was discussed in \Cref{sec:eff_ret_comp}. We compare these methods to five other retractions.
\\
$\bullet$ \textbf{The polar factor retraction:} 
The PF retraction can be $t$-parameterized by computing an SVD $\xi = \Gamma \Sigma V^T$. This yields $(I_p+t^2\xi^T\xi)^{-\frac{1}{2}}=V(I_p+t^2\Sigma^2)^{-\frac{1}{2}}V^T$. Hence, we can avoid computing a dense matrix square root and --inverse for each $t$, and we only need to compute the SVD of $\xi$ once. In the implementation, we only have to apply the square root- and the inverse operations on the diagonal of $(I_p+t^2\Sigma^2)^{-\frac{1}{2}}$.  Thus $R^{\textsf{PL}}_{\hat U}(t\xi)=(\hat U + t\xi)V(I_p+t^2\Sigma^2)^{-\frac{1}{2}} V^T$. 
\\
$\bullet$ \textbf{Quasi--geodesics:} \cite[Propostion 1]{Bendokat_quasiGeo2021} provides a formula for evaluations of a Grass\-mann--like quasi geodesic (QG Gr.). In the setting of geodesic interpolation, we can alternatively obtain a matrix $S=\begin{bmatrix}
        A&B^T\\
        B&C
    \end{bmatrix}\in\Skew(2p)$ and $Q\in \St(n,p)$ so that the curve $\rho(t)=\begin{bmatrix}
        \hat U&Q
    \end{bmatrix}\exp_m\left(tS\right)\begin{bmatrix}
        I_p\\
        0   
    \end{bmatrix}$ connects two points $\hat U, U$ \cite[Algorithm 2]{Bendokat_quasiGeo2021} (QG St.). The terms featuring in the formula can be precomputed. Replacing the exponential term with the Cayley map, and using the inverse Cayley map instead of $\log_m$ in Step 8 of \cite[Algorithm 2]{Bendokat_quasiGeo2021} yields a Cayley accelerated variant (QG St. Cay). 
\\
$\bullet$ \textbf{QR retraction:} We also include the QR retraction (QR), whose inverse can be computed by solving a sequence of $p$ linear equation systems and is provided in \cite[Algorithm 1]{Kaneko:2013}. 
A $t$--parametrization is not obvious, as we have to compute the unqiue QR decomposition $U+t\xi=QR$.  
\\
$\bullet$ \textbf{Cayley retraction:} The Cayley retraction cannot be $t$-parameterized in the same way as the polar- or polar-light retraction, since the matrices $B^TB$ and $A$ that appear under the inverse in \eqref{eq:Cayley_retraction_p} do not commute in general. We continue to use formula \eqref{eq:Cayley_retraction_p} to compute $R^{\textsf{Cay}}$.
\\
For each retraction $R$, we measure the computation time of obtaining $\xi=R^{-1}_{U_0}(U_1)$, as well as, where applicable, the time it takes to precompute relevant matrix decompositions and relevant matrix--matrix products (preprocessing). We then measure the total time it takes to compute the set of points $\{R_{U_0}(t_k\xi)\}_{k=0}^{N-1}$ using their $t$--dependent formulations (interpolation), where we take $N=11$. 
\begin{figure}[!ht]
    \centering
    \includegraphics[width = .9\textwidth]{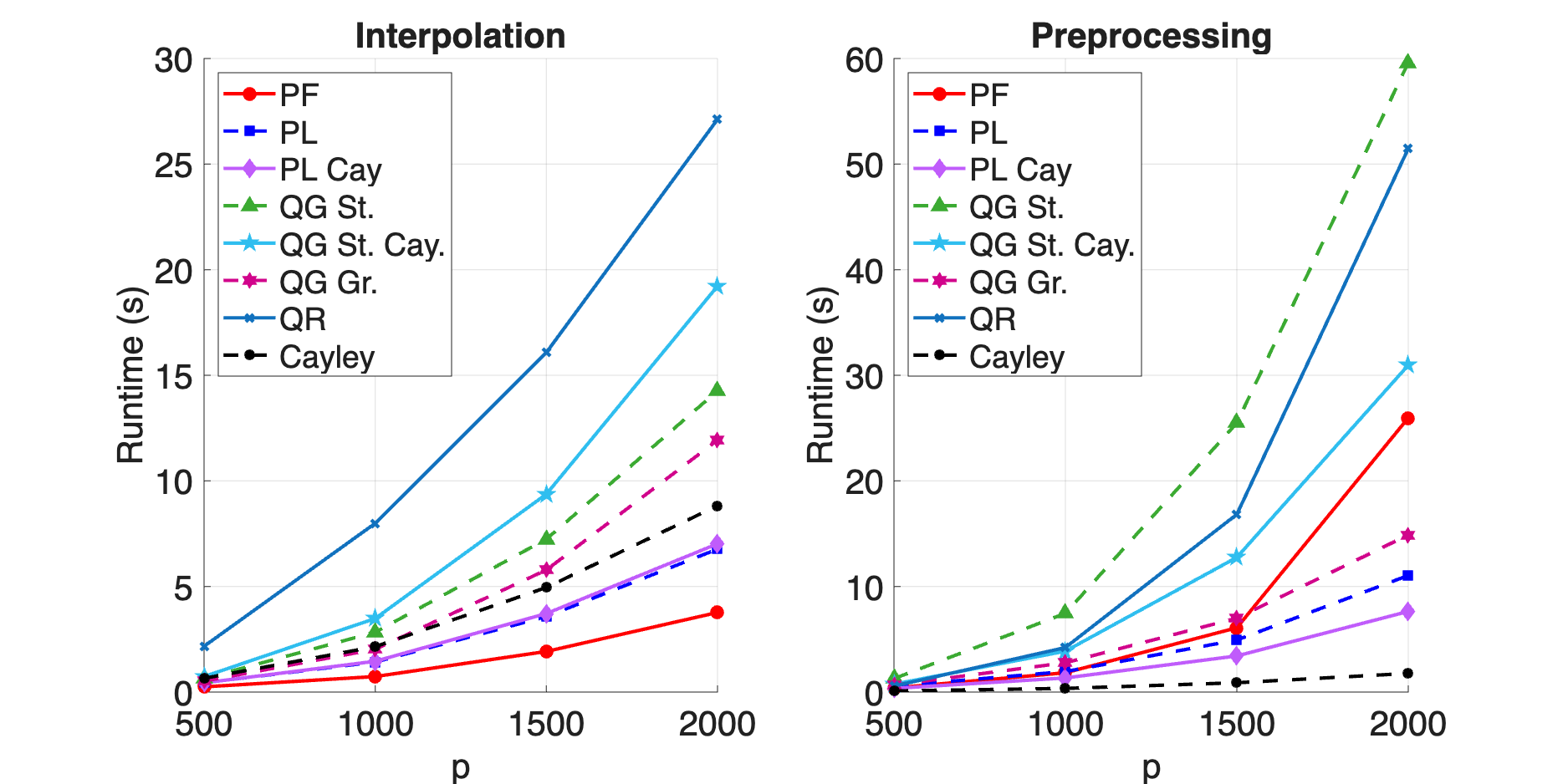}
    \caption{Timing results for computing the retractions along a fixed tangent vector $\xi$ with different scaling $t$ for increasing dimension $p$, mimicking the process of linear interpolation (left). The preprocessing costs are also displayed (right).
    }
    \label{fig:runtime_along}
\end{figure}
In \Cref{fig:runtime_along} we present the results on $\St(n,p)$ with $n=10000$ and $p\in \{500,1000,1500,2000\}$. We observe that the computational cost of the preprocessing step increases with $p$, similar to what we observed in \Cref{fig:timefig}. We see that the preprocessing time for computing $R^{\textsf{Cay}}$ is smaller than that of all the other retractions, and that the combined cost of interpolation and preprocessing at $p=2000$ is also smallest for $R^{\textsf{Cay}}$, despite the apparent higher cost of computing the points $\{R^{\textsf{Cay}}_{U_0}(t_k\xi)\}_{k=0}^{N-1}$. When ignoring the preprocessing step, $R^{\textsf{PF}}$ is most efficient in the interpolation step. $R^{\textsf{Cay}}$ fails to be as efficient as e.g. $R^{\textsf{PF}}$, since a dense ($p\times p$) matrix has to be formed and inverted for each $t_k$. It is noteworthy that computing the points $\{R^{\textsf{PL}}_{U_0}(t_k\xi)\}_{k=0}^{N-1}$ is faster than computing $\{R^{\textsf{PL Cay}}_{U_0}(t_k\xi)\}_{k=0}^{N-1}$. This is due to the efficient computation of the matrix exponential of a block--diagonal matrix.
We observe that the preprocessing time for $R^{\textsf{PL}}$ and $R^{\textsf{PL Cay}}$ ranks 2 and 3, respectively, in terms of runtime.


%

\subsection{Interpolation of POD bases}\label{sec:exp_pod}

As a practical example, we consider the interpolation of orthonormal bases obtained by computing the proper orthogonal decomposition of snapshot data from numerical integration of a partial differential equation (so--called POD bases). POD bases arise in the context of model order reduction \cite{MOR2008}, where they encode the essential dynamics of the system of interest. 
An important point is that neither the data nor the application dictates a specific metric. It is therefore a-priori unclear, which retraction to choose.
We consider the Fisher--KKP equation with a diffusion term and Dirichlet boundary conditions
\begin{equation}\label{eq:Fishereq}
    \partial_t u(t,x)=\partial_{xx}u(t,x)+f(u(t,x);\rho), \ \ (t,x)\in [0,T]\times [-L,L],
\end{equation}
where $0<\rho<1$ and $f(u;\rho)=\rho u(1-u)$ is a nonlinear function \cite{Adomian:1995}.

After a finite--differences approximation of the spatial term $\partial_{xx}$ we consider the following forward Euler scheme 
\begin{equation}\label{eq:Fisher_Euler}
    u(t_{n+1})=u(t_{n},x)+h(D_{xx}u(t_n)+f(u(t_n);\rho)), 
\end{equation}
where $f$ is evaluated entry--wise on the vector $u(t)\in \R^{N_x}$ and the matrix $D_{xx}\in \R^{N_x\times N_x}$ stems from the finite--differences approximation.
We consider the initial value $u(0,x)=e^{-x^2/2}$ and evolve the system \cref{eq:Fisher_Euler} with $T=10, L=30, h = 10^{-3}$ and $N_x=100$, for $\rho\in \{0.1,0.5,0.9\}$. For each $\rho$, we obtain a POD basis from the snapshot matrix $Y^{(\rho)}=[u(t_0),\dots,u(t_{N_t-1})]$ by computing the $p$ dominant left--singular vectors $U^{(\rho)}\in\St(N_x,p)$ of $Y^{(\rho)}$. 

We would like to highlight that computing an analytic path of an SVD is challenging even when the singular values are distinct \cite{BunseGerstner1991}. Algorithms for computing the SVD of a matrix curve $X(t)$ may additionally introduce discontinuities in sampled singular vectors for $X(t_0)$ and $X(t_0+\varepsilon)$, due to sign switches. This issue can be overcome as in \cite{ZimmermannHermite_2020}: For a curve $X(t)=U(t)\Sigma(t)V(t)^T$ we fix $t_0$ and for each $t$ we compute $S=\textnormal{sign}(\diag(U(t)^T U(t_0)))$ and replace $U(t)$ with $U(t)S$ and $V(t)$ with $V(t)S$. In the context of piecewise--linear interpolation, one can select one of the two endpoints as reference for each interval under consideration. Yet, in the experiment at hand, no sign adjustement was necessary.

We perform piecewise--linear interpolation. For a pair of data points $U^{(\rho_i)}, U^{(\rho_{i+1})}$, we compute the Riemannian logarithm $\xi_{\rho_i}=\Log_{U^{(\rho_{i})}}(U^{(\rho_{i+1})})$ or the inverse retraction $\xi_{\rho_i}=R_{U^{(\rho_i)}}^{-1}(U^{(\rho_{i+1})})$ and for $\rho^*\in [\rho_i,\rho_{i+1}]$ we evaluate $\Exp_{U^{(\rho)}}((\frac{\rho^*-\rho_i}{\rho_{i+1}-\rho_i})\xi_{\rho_i})$ or $R_{U^{(\rho)}}((\frac{\rho^*-\rho_i}{\rho_{i+1}-\rho_i})\xi_{\rho_i})$, which yields the interpolated POD basis. The relative errors are shown in \Cref{fig:fisher_exp}, from which it can be seen that the polar--light retraction and Riemannian exponential are close. The smallest interpolation error is obtained using the Grassmann--like quasi geodesics.  The polar factor retraction, while being locally second--order accurate, deviates significantly from the true data in comparison with the other methods. This is is most likely the result of instability of the Lyapunov equation one has to solve \cite[Equation 19]{Kaneko:2013}, since we obtain tangent vectors $\xi_{\rho_i}=(R^{\textsf{PF}}_{U^{(\rho_i})})^{-1}(U^{(\rho_{i+1})})$ of large norm, relative to what we see for the other retractions.  To check for analytic dependence on $\rho$, one can plot the smallest singular value together with the immediate adjacent singular values of the snapshot matrix obtained from the scheme \cref{eq:Fisher_Euler}, and observe exponential-like growth as $\rho$ increases. 

As the experiment involves matrices of small dimension, we have chosen to leave out timing information, and we refer the reader to \Cref{sec:experiments_timing} for a discussion involving larger matrices. 

\begin{figure}[!ht]
    \centering
    \includegraphics[width=0.8\linewidth]{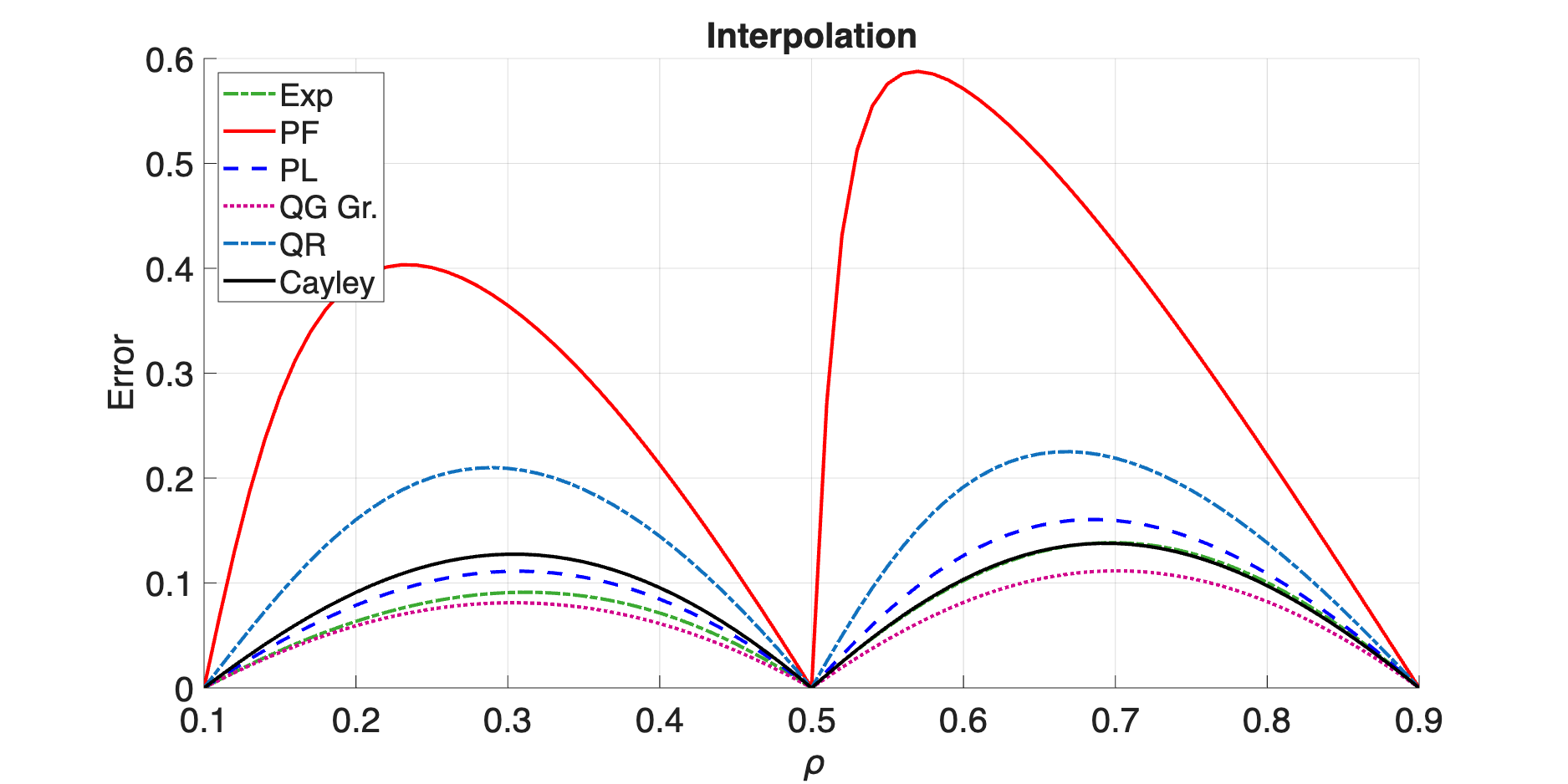}
    \caption{Relative errors associated to interpolating POD bases on $\St(N_x,p)$ for $\rho\in[0.1,0.9]$, computed as $\textnormal{error}=\frac{\| U_{\textnormal{int}}-U_{\textnormal{true}}\|_F }{\|U_{\textnormal{true}}\|_F}$. We observe that the polar--light retraction and Riemann normal coordinates perform similarly, while the polar factor retraction has higher error.}
    \label{fig:fisher_exp}
\end{figure}

\subsection{Computing the Riemannian barycenter of a geodesic triangle}\label{sec:exp_RBC}

The Riemannian center of mass (sometimes called the Fr\'echet-- or Karcher mean) is the manifold generalization of computing the mean of a dataset and was studied in \cite{Karcher1977}. Given a set of points $\{U_i\}\subset \St(n,p)$, the Riemannian center of mass $U_\mu$ is the solution of the following optimization problem
\begin{equation}\label{eq:RBC}
    U_\mu = \argmin_{U\in \St(n,p)}f(U)=\argmin_{U\in \St(n,p)}\frac{1}{2N}\sum_{i=1}^N \|\Log_U(U_i)\|^2.
\end{equation}
In contrast to the problem from \Cref{sec:exp_pod},  the formulation of the Riemannian barycenter inherently depends on the choice of a specific metric.
Hence, one may expect that, when the metric is, for example, Euclidean, retractions adapted to the Euclidean metric will perform better.
Solving \eqref{eq:RBC} by means of Riemannian steepest descent was discussed in \cite{AfsariTronVidal2013}, and it is guaranteed that a solution exists and is unique, provided that all data lie in an open ball $B_r(U')\subseteq \St(n,p)$ with $r\leq \frac{\pi}{2}$ \cite[Theorem 2.1]{Afsari2011}\footnote{In general, for a complete Riemannian manifold $(\mcM,g)$, the radius $r$ is bounded by $\frac{1}{2}\min\{\textnormal{inj}(\mcM),\frac{\pi}{\sqrt{\Delta}}\}$, where $\textnormal{inj}(\mcM)$ is the injectivity radius and $\Delta$ is the lower bound on the sectional curvature, where $\frac{1}{\sqrt{\Delta}}:=\infty$ if $\Delta<0$. For the Stiefel manifold equipped with the Euclidean metric, $\Delta =1$ is the sharp upper bound on the sectional curvature \cite{zimmermannstoye:2024}, provided $p\geq 2,n\geq p + 2$, and the injectivity radius is $\pi$ \cite{zimmermann2025injectivity}.}, but in practice larger domains of convergence are often observed.

The Riemannian gradient {\color{black} under the Euclidean metric} is given by \cite[Theorem 1.2]{Karcher1977}
\begin{equation}\label{eq:RBC_grad}
    \textnormal{grad}f(U)=-\frac{1}{N}\sum_{i=1}^N\Log_U(U_i).
\end{equation}
Computing \eqref{eq:RBC_grad} requires numerous evaluations of the Riemannian logarithm, which can be computationally costly. In this experiment we replace $\Log_U$ with an inverse retraction $R^{-1}$, and compute a barycenter of a geodesic triangle using Riemannian steepest descent, as outlined in \Cref{alg:SD_RBC}.\\
\rev{A word of caution: The gradient \eqref{eq:RBC_grad} depends on the chosen metric --  through the Riemannian logarithm. The logarithms under the Euclidean metric is different from the logarithm under the canonical metric. Replacing the Riemannian logarithm with an inverse retraction thus results in an approximation of the true gradient. The convergence criterion of step 2 of \Cref{alg:SD_RBC} may therefore indicate a near-perfect zero norm for the approximated gradient while the actual gradient is non-zero. Hence, the algorithm can only find an approximation to the true Riemannian barycenter, even when perfectly converged. The choice of retraction has therefore a non-negligible effect on the computation of barycenters.}\\
Whenever we replace $\Log$ with an inverse retraction $R^{-1}$, we use the corresponding retraction $R$ to compute the next iterate (Step 5). For more details on Riemannian optimization, see the standard textbooks \cite{AbsilMahonySepulchre2008,boumal2023}.

\begin{algorithm}[!ht]
\footnotesize
\caption{Steepest descent method for computing $U_\mu$ in \eqref{eq:RBC}}
\label{alg:SD_RBC}
    \begin{algorithmic}[1]
    \REQUIRE Data $\{U_i\}\subset \St(n,p)$, initial guess $U_0$, retraction and inverse retraction $R$, $R^{-1}$, tolerance $\tau>0$, {\color{black}initial step size $\alpha_0>0$, reduction factor $\delta\in (0,1)$, and constant $c\in(0,1)$}
    \STATE $k = 0$, $\alpha=\alpha_0$
    \WHILE{$\|\textnormal{grad}f(U)\|_F>\tau$}
    \STATE Compute $\textnormal{grad} f(U_k)$ according to \eqref{eq:RBC_grad}, with $R^{-1}_{U_k}$ in place of $\Log_{U_k}$
    {\color{black}\IF{$f(R_{U_k}(-\alpha\textnormal{grad} f(U_k))\leq f(U_k)+c\|\textnormal{grad}f(U_k)\|_F^2$}
    \STATE $U_{k+1}=R_{U_k}(-\alpha \textnormal{grad} f(U_k) )$
    \STATE $k = k + 1$, $\alpha=\alpha_0$ 
    \ELSE
    \STATE $\alpha \leftarrow \delta \alpha$ \COMMENT{Overwrite $\alpha$ with $\delta \alpha$}
    \ENDIF}
    \ENDWHILE
    \RETURN $U_\mu=U_{k+1}$ 
\end{algorithmic}

\end{algorithm}

We work under the Euclidean metric. To set up the experiment, let $U_0\in \St(n,p)$ and generate two random tangent vectors $\xi_1,\xi_2\in T_{U_0}\St(n,p)$ so that $\|\xi_1\|_F=\|\xi_2\|_F=0.8\pi$, and let $U_1=\Exp_{U_0}(\xi_1), U_2=\Exp_{U_0}(\xi_2)$ cf. \Cref{fig:geod_triangle}. 
$\xi_i=U_0A+(I_n-U_0U_0^T)B$ is generated so that it has a nonzero $B$--block. 
\begin{figure}[t!]
    \centering
    \includegraphics[width=0.6\textwidth]{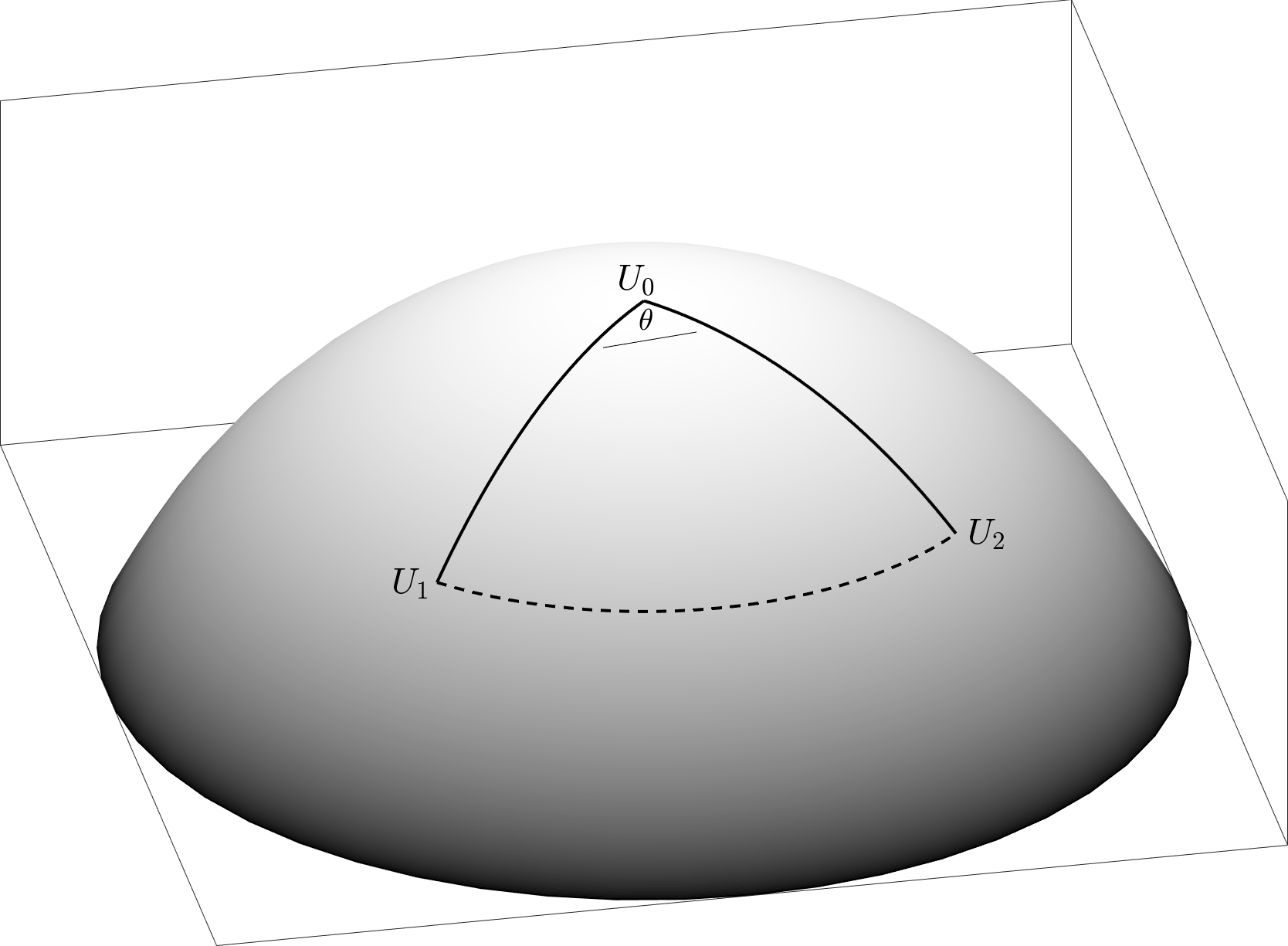}
    \caption{Schematic picture of a geodesic triangle on a manifold. The sides $U_0\to U_1$ and $U_0\to U_2$ have the
same lengths. We use a gradient descent to compute the Riemannian center
of mass.}
    \label{fig:geod_triangle}
\end{figure}
For each retraction we obtain a barycenter $U_{\mu,R}$, which we then compare to the Riemannian barycenter $U_{\mu}$ obtained by solving the optimization problem \eqref{eq:RBC} using the Riemannian exponential and --logarithm maps. 

{\color{black}We run \Cref{alg:SD_RBC} with tolerance $\tau=10^{-10}$ and measure the computational cost (in seconds) required to compute the Riemannian gradient and to subsequently compute the next iterate $U_{k+1}=R_{U_k}(-\alpha \textnormal{grad} f(U_k))$ (steps 3 -- 9), with an Armijo backtracking scheme with $c=10^{-4}$, initial step size $\alpha = \frac{1}{4}\pi$, and reduction factor $\delta = 0.5$. For efficiently evaluating the retractions as $\alpha$ is shrunken, we use the $t$-parameterizations similar to what we did in \Cref{sec:experiments_timing}, and we precompute relevant quantities for efficient evaluations of the the Riemannian exponential map as well.}

\Cref{fig:timings_RBC} displays the results of applying Riemannian steepest descent with various retractions and their inverses replacing the Riemannian logarithm in \eqref{eq:RBC_grad}. It is seen that using either of the retractions is significantly less costly than using the Riemannian exponential and --logarithm maps. \Cref{tab:RBC_tab} contains the results of the same experiment, as well as results of two additional experiments in smaller dimensions, where it is seen that the retractions in general allow for substantial reductions in computational costs, while the obtained barycenters remain relatively close to the one obtained using the Riemannian exponential and --logarithm. Using $R^{\textsf{PL Cay}}$ instead of $R^{\textsf{PL}}$ leads to a slight decrease in accuracy, while being significantly faster. 

Among the considered methods, the accuracy is lowest for the Cayley retraction. \rev{Throughout the set of numerical experiments, it is consistently one order of magnitude lower than under the polar-factor retractions.} This showcases that the \rev{approximation order of the retraction with respect to the selected metric} matters: The Cayley retraction is not a second--order retraction under the Euclidean metric. When running the same experiment under the canonical metric ($\beta = \frac{1}{2}$ in \eqref{eq:St_exp}) (not shown),  the Cayley retraction is found to be most beneficial in terms of both accuracy and speed. 
\begin{figure}
    \centering
    \includegraphics[width=0.8\linewidth]{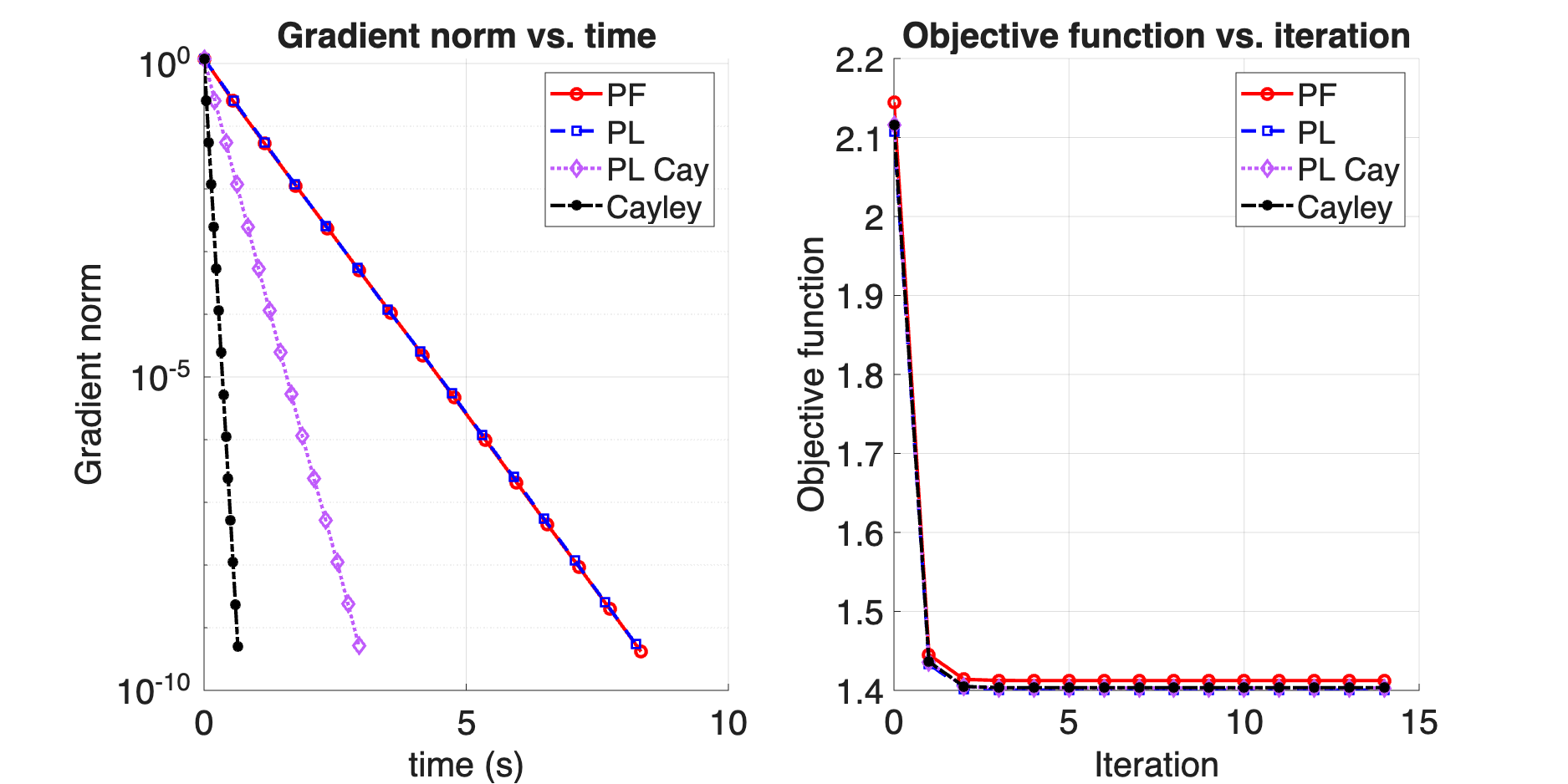}
    \caption{Results of applying Riemannian steepest descend to the problem \eqref{eq:RBC} on $\St(1500,400)$ to compute the Riemannian barycenter of a geodesic triangle. The results are also shown in the bottom of \Cref{tab:RBC_tab}.}
    \label{fig:timings_RBC}
\end{figure}

\begin{table}[htbp]
    \footnotesize
    \centering
    {\color{black}\begin{tabular}{|l | c r r r|}\hline
         & $\theta$ & Iter. & Time (s.) & $\|U_\mu-U_{\mu,R}\|_F$\\ \hline 
        $\St(200,30)$ & & & & \\
        Exp  & \multirow{4}{*}{1.54}& 16 &  0.89 & --\\
        PF   & & 13 & 0.024 & 0.026 \\
        PL   & & 16 & 0.144 &  0.005 \\
        PL Cay & & 16 & 0.023 & 0.008 \\ 
        Cayley & & 14 & 0.005 & 0.077 \\ \hline
        $\St(700,200)$ & & & & \\
        Exp & \multirow{4}{*}{1.56}&14 &  18.34 &-- \\
        PF   & &14 & 1.94  & 0.003 \\
        PL   & &14 & 2.16 & 7.12e-4\\
        PL Cay & &14 & 0.65 & 0.001\\ 
        Cayley & & 14 & 0.09 & 0.028 \\\hline 
        $\St(1500,400)$ & & & & \\
        Exp &\multirow{4}{*}{1.57} &19 &  234.34& -- \\
        PF   & &14 & 8.34 &  0.002 \\
        PL   & &14 & 8.28 & 3.52e-4\\
        PL Cay & &14 & 2.97 & 6.14e-4 \\ 
        Cayley & & 14 & 0.64 & 0.02 \\ \hline 
    \end{tabular}}
    \caption{Results of running Riemannian steepest descent for various dimensions. $\theta=\arccos(\frac{\langle \xi_1,\xi_2\rangle_1}{\|\xi_1\|_F\|\xi_2\|_F})$ is the angle between the generated tangent vectors used to obtain the two corner points $U_1$ and $U_2$\rev{, see \Cref{fig:geod_triangle}}. $U_{\mu}$ is the Riemannian barycenter obtained using \cref{alg:SD_RBC} without replacing the Riemannian logarithm, and $U_{\mu,R}$ is the barycenter obtained using one of the retractions. {\color{black} For reference, we also display the number of iterations and runtime for computing the reference barycenter $U_{\mu}$}} 
    \label{tab:RBC_tab}
\end{table}

\section{Conclusions and future work}
\label{sec:conclusions}
By a modification of the classical polar factor retraction on the Stiefel manifold, we introduce a new Stiefel retraction, the polar-light retraction, which is second--order accurate under the Euclidean metric and admits a closed-form inverse. 

The proposed polar--light retraction incurs a higher computational cost in forward evaluation compared to the classical polar factor retraction, mainly owing to an additional $(n\times p)$ times $(p\times p)$ matrix-matrix product in the factor $(U(\exp_m(A)-A) + \xi)$ in \eqref{eq:psi_U_Onp2_efficient}. However, its inverse can be computed more efficiently. This may be beneficial in logarithm--heavy computations, such as when computing Riemannian barycenters in large dimensions, as well as when performing Hermite interpolation, similar to \cite{jensen2025maxvol} (where the Grassmann manifold was considered). 

The formulas for the closed-form expression for the inverse polar-light retraction make it explicit how the independent parameters $A\in \Skew(p)$, $B\in \R^{(n-p)\times p}$ of a Stiefel tangent matrix $\xi = \mathbf{Q}\left[\begin{smallmatrix}
    A\\
    B
\end{smallmatrix}\right]$ enter the calculation.
Numerical experiments show that the
polar-light retraction is consistently closer to the associated Stiefel geodesics than the classical polar factor retraction. 
{\em In particular, the contribution of the $A$-block is captured more accurately by the polar‑light retraction.}
(For $A=0$, the polar-light retraction coincides with the full polar factor retraction.) To decrease computational costs, the matrix exponential and matrix logarithm functions  may be consistently replaced with their Cayley approximations without compromising the second-order property of the retraction. Moreover, in light of \Cref{sec:exp_pod}, the inverse polar factor retraction may degenerate quicker than the inverse of the polar-light retraction.

\rev{We have implemented all the retractions considered in \Cref{sec:experiments_timing}
in an efficient $t$--dependent realization with precomputed $t$--independent matrix quantities, except for the QR retraction.} Apart for interpolation tasks, this could also be beneficial in optimization, if line searching is conducted. It is an open question if this can be done for the QR retraction. 
\\
$\bullet$ If one aims for speed, our numerical experiments indicate that the Cayley retraction \eqref{eq:Cayley_retraction_p} offers an efficient choice, which is also second--order accurate under the canonical metric. The $t$--dependent implementation requires inverting a $p\times p$ matrix, and is less efficient when comparing to e.g. the polar-light retraction. \\
$\bullet$ If one aims to approximate the Riemannian exponential map under the Euclidean metric with high precision, the polar--light retraction \eqref{eq:psi_U_Onp2_efficient} or its Cayley-accelerated variant \eqref{eq:varphi_U_Onp2_cay} appear to provide an efficient choice, in particular, if evaluating the inverse retraction is also of interest. They can additionally be implemented for efficient $t$--dependent realizations. \\ 
$\bullet$ We would like to emphasize that we do not advocate to use the Euclidean metric or any other metric as a default choice. The choice of metric is problem-dependent.

\bibliographystyle{siamplain}
\bibliography{SurrogateModelling}
\end{document}